\documentclass[12pt]{amsart}
\usepackage{mathpazo}
\usepackage{hyperref}
\usepackage{bm}
\usepackage{verbatim}
\usepackage{slashbox}
\usepackage[margin=1in]{geometry}
\usepackage{array,multirow}
\usepackage{graphicx}

\usepackage[all]{xy}
\usepackage{array}

\usepackage{amssymb,amsmath,latexsym,amsthm}
\newtheorem{theorem}{Theorem}[section]
\newtheorem{lemma}[theorem]{Lemma}
\newtheorem{proposition}[theorem]{Proposition}
\newtheorem{corollary}[theorem]{Corollary}
\newtheorem{predefinition}[theorem]{Definition}
\newenvironment{definition}{\begin{predefinition}\rm}{\end{predefinition}}
\newtheorem{preremark}[theorem]{Remark}
\newenvironment{remark}{\begin{preremark}\rm}{\end{preremark}}
\newtheorem{prenotation}[theorem]{Notation}
\newenvironment{notation}{\begin{prenotation}\rm}{\end{prenotation}}
\newtheorem{preexample}[theorem]{Example}
\newenvironment{example}{\begin{preexample}\rm}{\end{preexample}}
\newtheorem{preclaim}[theorem]{Claim}

\newtheorem{prequestion}[theorem]{Question}

\newtheorem{preapplication}[theorem]{Application}

\theoremstyle{remark}

\numberwithin{equation}{section}

\newcommand \ZZ {{\mathbb Z}}
\newcommand \QQ {{\mathbb Q}}

\newcommand \NN {{\mathbb N}}
\newcommand  \FF {{\mathbb F}}
\newcommand \GG {{\mathbb G}}
\newcommand \HH {{\mathbb H}}
\newcommand \EE {{\mathbb E}}
\newcommand \PP {{\mathbb P}^1}

\newcommand  \Frob {{\boldsymbol F}}
\newcommand  \Ver {{\boldsymbol V}}

\global\let\hom\undefined
\DeclareMathOperator{\hom}{Hom}
\global\let\ker\undefined
\DeclareMathOperator{\ker}{Ker}

\DeclareMathOperator{\Res}{Res}

\DeclareMathOperator{\GL}{GL}
\DeclareMathOperator{\Lie}{Lie}
\DeclareMathOperator{\disc}{disc}
\DeclareMathOperator{\perf}{perf}

\newcommand \CA {{\mathcal A}}
\newcommand \CM {{\mathcal M}}

\newcommand \cO {{\mathcal O}}
\newcommand \cS {{\mathcal S}}
\newcommand \CT {{\mathcal T}}

\newcommand \Sh {{\rm Sh}}
\newcommand \codim {{\rm codim}}

\newcommand \CX {{\mathcal X}}

\newcommand \frf {{\mathfrak f}}

\newcommand \CO{{\mathfrak O}}
\newcommand \p{{\mathfrak p}}
\newcommand \fd{{\mathfrak d}}

\newcommand \CC {{\mathbb C}}

\newcommand \RR {{\mathbb R}}

\newcommand \hh {{\mathfrak h}}

\newcommand \cf {{\mathfrak f}}

\newcommand \bA {{\mathbb A}}

\newcommand \QM {{\QQ[\mu_m]}}

\newcommand \sss {{\rm ss}}
\newcommand \co {{\mathfrak o}}

\newcommand \CJ {{\mathcal J}}

\newcommand{\bfp}{{\overline{\mathbb F}_p}}

\newcommand{\E}{{E}}

\usepackage[usenames,dvipsnames]{color} 

\title[Newton polygons from special families of cyclic covers]{Newton polygons arising from special families of cyclic covers of the projective line}
\date{}

\author{Wanlin Li}
\address{Department of Mathematics,
University of Wisconsin,
Madison, WI 53706, USA}
\email{wanlin@math.wisc.edu}

\author{Elena Mantovan}
\address{Department of Mathematics,
California Institute of Technology
Pasadena, CA 91125, USA}
\email{mantovan@caltech.edu}

\author{Rachel Pries}
\address{Department of Mathematics, 
Colorado State University, 
Fort Collins, CO 80523, USA}
\email{pries@math.colostate.edu}

\author{Yunqing Tang}
\address{Department of Mathematics,
Princeton University,
Princeton, NJ 08540, USA}
\email{yunqingt@math.princeton.edu}

\begin{document}

\begin{abstract}
By a result of Moonen, there are exactly 20 positive-dimensional families of cyclic covers of the projective line
for which the Torelli image is open and dense in the associated Shimura variety.
For each of these, we compute the Newton polygons, and the $\mu$-ordinary Ekedahl--Oort type, 
occurring in the characteristic $p$ reduction of the Shimura variety.
We prove that all but a few of the Newton polygons appear on the open Torelli locus.
As an application, we produce multiple new examples of Newton polygons and Ekedahl--Oort types
of Jacobians of smooth curves in characteristic $p$.
Under certain congruence conditions on $p$, these include:
the supersingular Newton polygon for genus $5,6,7$;
fourteen new non-supersingular Newton polygons for genus $5-7$; 
eleven new Ekedahl--Oort types for genus $4-7$ and,
for all $g \geq 6$, the Newton polygon with $p$-rank $g-6$ with slopes $1/6$ and $5/6$.

MSC10: primary 11G18, 11G20, 11M38, 14G10, 14G35; secondary 11G10, 14H10, 14H30, 14H40, 14K22


 










\end{abstract}

\maketitle

\section{Introduction}

A fundamental problem in arithmetic geometry is to determine which abelian varieties arise as Jacobians of (smooth) curves. This question is equivalent to studying (the interior of) the Torelli locus in certain Siegel varieties.  
For $g=1,2,3$, the Torelli locus is open and dense in the moduli space 
${\mathcal A}_g$ of principally polarized abelian varieties of 
dimension $g$.
For any family $T$ of cyclic covers of the projective line, 
Deligne and Mostow construct the smallest PEL-type Shimura variety containing the image of $T$ under the Torelli morphism \cite{deligne-mostow}.  
In \cite{moonen}, Moonen shows that there are precisely twenty positive-dimensional families of cyclic covers of 
${\mathbb P}^1$ for which {the image of $T$ 
under the Torelli morphism} is open and dense in the associated Shimura variety; these families are called special. 

In positive characteristic $p$, the $p$-rank, Newton polygon, and Ekedahl--Oort type are 
discrete invariants of an abelian variety.  
It is a natural question to ask which of these invariants can be realized by the Jacobian of a smooth curve.  
As each discrete invariant yields a stratification of the reduction modulo $p$ of 
the Siegel variety, 
this question is equivalent to understanding which strata intersects the interior of the Torelli locus.  

Ultimately, the goal is to understand the geometry of the induced stratifications of the Torelli locus 
(e.g., the connected components of each stratum and their closures), in the same way 
that the geometry of these stratifications is understood for Siegel varieties.
For example, for any prime $p$, genus $g$ and integer $f$ such that $0 \leq f \leq g$, Faber and van der Geer prove in \cite{FVdG}
that the $p$-rank $f$ stratum is non-empty and has the expected codimension $g-f$ in the moduli space ${\mathcal M}_g$ of curves of genus $g$.
See also \cite{AP:mono}, \cite{AP:hyp}, and \cite{glasspries}.  

Beyond genus $3$, much less is known about the Newton polygon, 
more precisely the Newton polygon of the characteristic polynomial of Frobenius.  
In \cite[Expectation 8.5.4]{oort05}, for $g \geq 9$,
using a dimension count, Oort observed that the interior of the 
Torelli locus is unlikely to intersect the supersingular locus or other 
Newton polygon stratum of high codimension in ${\mathcal A}_g$.
This suggests that it is unlikely for all Newton polygons to occur for Jacobians of smooth curves for each prime $p$.

In our earlier work \cite{LMPT}, we studied Newton polygons of cyclic covers of the projective line branched at $3$ points.
In this case, the Jacobian is an abelian variety with complex multiplication and its Newton polygon can be computed using the Shimura--Taniyama theorem. 
We used this to find several new Newton polygons having $p$-rank $0$ 
which occur for Jacobians of smooth curves.

The Newton polygon stratification 
of PEL-type Shimura varieties is well understood 
by the work of Viehmann and Wedhorn \cite{viehmann-wedhorn}, based on earlier work of Kottwitz, Rapoport, and Richartz \cite{kottwitz-rapoport}, \cite{kottwitz1}, \cite{kottwitz2}, \cite{rapoport-richartz}.
In this paper, we demonstrate this theory by determining all the Newton polygons of Jacobians that occur for the 20 special families of cyclic covers from Moonen's paper \cite{moonen}.  Furthermore, using \cite{moonenST},
we compute the $\mu$-ordinary mod $p$ Dieudonn\'e modules (e.g., Ekedahl--Oort types) for the 20 special families.

We then investigate which Newton polygons and Dieudonn\'e modules on these lists arise for a smooth curve in the family. 
We conclude affirmatively in three distinct cases: \\
(i) for the $\mu$-ordinary Newton polygon and Dieudonn\'e module, Proposition~\ref{muord};\\
(ii) for any PEL-indecomposable Newton polygon, Proposition \ref{indecomp}; and \\
(iii) for the Newton polygon of the closed stratum (called the basic Newton polygon) if the family has dimension $1$
and if $p$ is sufficiently large, Proposition~\ref{basic}.

For (iii), we refer to Theorem \ref{prop_irredcomp} in the Appendix, where we prove, 
under some restrictions on the prime $p$,
that the number of geometrically irreducible components of the basic locus at $p$ of certain unitary Shimura variety grows to infinity with $p$. 
More precisely, Theorem \ref{prop_irredcomp} holds for unitary Shimura varieties of signature $(1,n-1)$ at one real place 
and definite signature at all other real places. In particular, our results hold for the simple Shimura varieties considered by Harris--Taylor in \cite{HT}. 
Earlier instances of such a statement are due to Vollaard \cite{vollaard}, and Vollaard--Wedhorn \cite{vollaardwedhorn}
who prove a similar result under the further assumption that the implicit CM field defining the unitary group is quadratic-imaginary, and to Liu--Tian--Xiao--Zhang--Zhu \cite {XZ2} at primes which are totally split in the CM field. Here, we only exclude the case of primes which are inert in the quadratic imaginary extension for $n$ even. 
We regard Theorem \ref{prop_irredcomp} to be of independent interest.

We calculate that these three criteria together with the de Jong--Oort purity result 
imply that all but a few of the Newton polygons for Moonen's special families 
occur for Jacobians of smooth curves. 
More precisely, we prove the following result.

\begin{theorem} (See Theorem \ref{onemore})
Let $(m,N,a)$ denote the monodromy datum for one of Moonen's special families from \cite[Table 1]{moonen}.  Assume $p \nmid m$.   
Let $\nu \in \nu(B(\mu_m, \cf))$ be a Newton polygon occuring on the closure in ${\mathcal A}_g$
of the image of the family under the Torelli morphism. Then $\nu$
occurs as the Newton polygon of the Jacobian of a smooth curve in the family unless either:
\begin{enumerate}
\item the dimension of the family is one, $\nu$ is supersingular, and $p$ is not sufficiently large; or
\item the dimension of the family is at least two and $\nu$ is supersingular.
\end{enumerate}
\end{theorem}

In cases (1) and (2), we also expect that there exists a smooth curve in the family which is supersingular;
(we prove this for large $p$ in \cite[Theorem 6.1]{LMPT3}).

We deduce the following applications.

First, we verify the existence of supersingular curves of genus $5-7$
for primes satisfying certain congruence conditions.  See \cite[Theorem 1.1]{LMPT} for a complementary result.

\begin{theorem}[See Theorem \ref{Tapp1}]\label{Tintro1}
There exists a smooth supersingular curve of genus $g$ defined over $\overline{\FF}_p$ 
for all sufficiently large primes satisfying the given congruence condition:
\\
$g=5$ when $p \equiv 7 \bmod 8$; $g=6$ when $p \equiv 2,3,4 \bmod 5$; and $g=7$ when $p \equiv 2 \bmod 3$.
\end{theorem}

The second application is in Theorems \ref{TmuordNP} and \ref{TbasicNP}: 
under certain congruence conditions on $p$, 
we verify that 14 new non-supersingular Newton polygons and 11 new Dieudonn\'e modules occur 
for Jacobians of smooth curves of genus $4-7$.

Every abelian variety is isogenous to a factor of a Jacobian.  
Thus every $\lambda \in [0,1] \cap {\mathbb Q}$ is known to occur as a 
slope for the Newton polygon of the Jacobian of a smooth curve. 
Usually, however, there is no control over the other slopes.
In the third application, for $p \equiv 2,4 \bmod 7$ and $g \geq 6$, we show that the slopes $1/6$ and $5/6$ occur 
for a smooth curve of genus $g$ defined over $\overline{\FF}_p$ with the other slopes of the Newton polygon 
all $0$ and $1$, Theorem~\ref{Tapp2}.

As a final application, we consider one non-special family of curves in Section \ref{Snonsp}.
Consider a prime $m \equiv 3 \bmod 4$ (with $m \not = 3$) and a prime $p$ which is a quadratic 
non-residue modulo $m$. 
In Theorem \ref{last}, we find smooth curves of genus $g_m=(m-5)(m-1)/4$ 
defined over $\overline{\FF}_p$ whose Newton polygon only has slopes $0$, $1/2$, and $1$
and such that the multiplicity of the slope $1/2$ is at least $2\sqrt{g_m}$.

In future work, we solve questions about Newton polygons for curves of arbitrarily large genera 
using a new induction argument for Newton polygons 
of cyclic covers of ${\mathbb P}^1$.
We use the Newton polygons from this paper as base cases in that induction process. 

\subsection*{Organization of the paper} \mbox{ } \\

Section \ref{sec_prelim} recalls basic definitions and facts about group algebras, families of cyclic covers of ${\mathbb P}^1$, Newton polygons, and mod $p$ Dieudonn\'e modules.

Section \ref{SShimura} contains information about Shimura varieties and positive-dimensional special families of cyclic covers
of the projective line.

In Section \ref{sec_posdim}, we study the Newton polygon stratification of PEL-type Shimura varieties. 
We demonstrate the general theory by computing the Newton polygons for several prototypical 
examples from \cite{moonen} in Section \ref{ScomputeNP}.  
In Section \ref{EO}, we determine the $\mu$-ordinary Dieudonn\'e modules (Ekedahl--Oort types) for the 
special families.  

Section \ref{sec_smooth} discusses which Newton polygon strata intersect the smooth Torelli locus.

Section \ref{sec_table} contains tables of data, 
including the Newton polygons and the $\mu$-ordinary Dieudonn\'e modules 
for each of the twenty special families of cyclic covers 
from \cite{moonen}.

Section \ref{Sapplication} contains the proofs of the main theorems. 

In the appendix, we compute a lower bound for the number of irreducible components of the basic locus of simple Shimura varieties.

\subsection*{Acknowledgements} \mbox{ } \\
This project began at the \emph{Women in Numbers 4} workshop at the Banff International Research Station.
Pries was partially supported by NSF grant DMS-15-02227. We thank Liang Xiao, Xinwen Zhu, and Rong Zhou for discussions about the appendix and thank Liang Xiao for the detailed suggestions on the writing of the appendix.  We would like to thank the referee for many helpful comments.

\section{Notation and background} \label{sec_prelim} 

\subsection{The group algebra $\QM$}\label{prelim_gpalg} \mbox{ } \\

For an integer $m\geq 2$, let $\mu_m:=\mu_m(\CC)$ denote the group of $m$-th roots of unity in $\CC$. 
For each positive integer $d$, we fix a primitive $d$-th root of unity $\zeta_d=e^\frac{2\pi i}{d}\in\CC$.
Let $K_d=\QQ(\zeta_d)$ be the $d$-th cyclotomic field over $\QQ$ of degree $\phi(d)$.

Let $\QM$ denote the group algebra of $\mu_m$ over $\QQ$.  
Then $\QM=\prod_{0<d\mid m}K_d$.
The involution $*$ on $\QM$ induced by the inverse map on $\mu_m$ preserves each cyclotomic factor $K_d$, and for each $d\mid m$, the restriction of $*$ to $K_d$ agrees with complex conjugation.

Set $\CT:={\rm Hom}_\QQ(\QM,\CC)$.
Each $(\QM\otimes_\QQ\CC)$-module $W$ has an eigenspace decomposition $W=\oplus_{\tau\in\CT} W_\tau$, where $W_\tau$ is the subspace of $W$ on which $a\otimes 1\in\QM\otimes_\QQ\CC$ acts as $\tau(a)$.
We fix an identification $\CT=\ZZ/m\ZZ$ by defining, for all $n\in \ZZ/m\ZZ$, \[\tau_n(\zeta):=\zeta^n, \text{ for all }\zeta\in\mu_m.\]
For all $n\in\ZZ/m\ZZ$, and $a\in\QM$, ${\tau}_{-n}(a)=\tau_{n}(a^*)$.
We write $\tau_n^*:=\tau_{-n}$.

For $\tau = \tau_n \in\CT$, the {\em order of} $\tau$ is the
order of $n$ in $\ZZ/m\ZZ$.
The homomorphism $\tau:\QM\to\CC$ factors via
$K_d$ if and only if $\tau$ has order $d$.

For each rational prime $p$, we fix an algebraic closure $\overline{\QQ}_p^{\rm alg}$of $\QQ_p$, 
and an identification $\iota: \CC\simeq \CC_p$, where $\CC_p$ denotes the $p$-adic completion of  $\overline{\QQ}_p^{\rm alg}$.  
Let
$\overline{\QQ}_p^{\rm un}$ denote the maximal unramified extension of $\QQ_p$ in  $\overline{\QQ}_p^{\rm alg}$, 
and by $\sigma$ the Frobenius of $\overline{\QQ}_p^{\rm un}$. 

If $p \nmid m$, then $\QM$ is unramified at $p$. 
Via $\iota$, we identify
$\CT={\rm Hom}_\QQ(\QM,\overline{\QQ}_p^{\rm un})$.
Hence, there is a natural action of $\sigma$ on $\CT$, defined by $\tau\mapsto \tau^\sigma:=\sigma\circ \tau$. 
Note that $\tau_n^\sigma=\tau_{pn}$ for all $n\in\ZZ/m\ZZ$. 
For each $\sigma$-orbit $\co\subseteq \CT$, the order of $\tau$ is the same for all $\tau\in \co$. 
We write $d_\co$ for the order of $\tau\in \co$.

Let $\CO$ be the set of $\sigma$-orbits $\co$ in $\CT$.  For each $\tau\in\CT$, let $\co_\tau$ denote
its $\sigma$-orbit. 
Each $\tau: \QM\rightarrow \overline{\QQ}_p^{\rm un}$ determines a prime $\p$ of $\QM$ above $p$, depending only on the orbit $\co_\tau$. Hence,
the set $\CO$ is in bijection with the set of primes $\p$ of $\QM$ above $p$. 
For an orbit $\co \in \CT$, 
we write $\p_\co$ for the associated prime above $p$ and
and $K_{d_\co, \p_\co}$ for the completion of $K_{d_\co}$ along the prime $\p_\co$.

\subsection{Families of cyclic covers of the projective line}\label{prelim_curve} \mbox{ } \\

We follow \cite[\S\S2,3]{moonen}.
Fix integers $m\geq 2, N\geq 3$ and an $N$-tuple of positive integers $a=(a(1),\dots, a(N))$.
Then $a$ is an {\em inertia type} for $m$ and  
$(m,N,a)$ is a {\em monodromy datum} if
\begin{enumerate}
\item $a(i)\not\equiv 0\bmod m$, for all $i=1, \dots, N$, 
\item $\gcd(m, a(1),\dots, a(N))=1$, and
\item $\sum_ia(i)\equiv 0 \bmod m$.
\end{enumerate}

Fix a monodromy datum $(m,N,a)$. 
Working over $\ZZ[1/m]$, let $U\subset ({\mathbb A}^1)^N$ be the complement of the weak diagonal, 
which consists of points where any two of the coordinates are the same. 
Consider the smooth projective (relative) curve $C$ over $U$ whose fiber at each point $t=(t(1),\dots, t(N))\in U$ has affine model 
\begin{equation} \label{EformulaC}
y^m=\prod_{i=1}^N(x-t(i))^{a(i)}.
\end{equation}
The function $x$ on $C$ yields a map $C \to {\mathbb P}^1_U$
and there is a $\mu_m$-action on $C$ over $U$ given by $\zeta\cdot (x,y)=(x,\zeta\cdot y)$ for all $\zeta\in\mu_m$.
Thus $C \to {\mathbb P}^1_U$ is a $\mu_m$-Galois cover. 

For a point $t\in U$, the cover $C_t\to \PP$ is a $\mu_m$-Galois cover, branched at $N$ points $t(1),\dots , t(N)$ in $\PP$, and 
with local monodromy $a(i)$ at $t(i)$.
By the hypotheses on the monodromy datum, the fibers of $C\to U$ are geometrically irreducible curves of genus $g$, where
\begin{equation} \label{Egenus}
g=g(m,N,a)=1+\frac{(N-2)m-\sum_{i=1}^N\gcd(a(i),m)}{2}.
\end{equation}

Let $\CM_g$ denote the moduli space of smooth projective curves of genus $g$
and let $\CA_g$ denote the Siegel moduli space of principally polarized abelian varieties of dimension $g$ over $\ZZ$.\footnote{When we talk about universal families on the stacks, it means that we pass to a suitable level structure for such universal objects to exist.}
The composition of the Torelli map with the morphism $U\to {\mathcal M}_g$ defined by the curve $C\to U$ yields a morphism over $\ZZ[1/m]$ denoted by
\[j=j(m,N,a):U\to {\mathcal M_g}\to \CA_g.\] 

\begin{definition} \label{DZmNa}
If $(m,N,a)$ is a monodromy datum, let $Z^0(m,N,a)$ be the image of $j(m,N,a)$ in $\CA_g$ and let
$Z(m,N,a)$ be its closure in $\CA_g$. 
\end{definition}

\begin{remark}
By definition, $Z(m,N,a)$ is a closed, reduced substack of $\CA_g$.  
It is also irreducible \cite[Corollary 7.5]{fultonhur}, \cite[Corollary 4.2.3]{wewersthesis}.

The substack $Z(m,N,a)$ depends uniquely on the equivalence class of the monodromy datum $(m,N,a)$, where 
$(m,N,a)$ and $(m',N',a')$ are equivalent if $m=m'$, $N=N'$, 
and the images of $a,a'$ in $(\ZZ/m\ZZ)^N$ are in the same orbit under
$(\ZZ/m\ZZ)^*\times {\rm Sym}_N$. 
\end{remark}

Fix a point $t\in U(\CC)$ and let $V$ be the first Betti cohomology group $H^1(C_t(\CC), \QQ)$. 
Then $V$ has a Hodge structure of type $(1,0)+(0,1)$ with the $(1,0)$ piece given by $H^0(C_t(\CC), \Omega^1_{C_t})$ 
via the Betti--de Rham comparison.   
Let $V^+$ (resp.\ $V^-)$ denote the $(1,0)$ (resp.\ $(0,1)$) piece. 

Then $V$, $V^+$, and $V^-$ are $\QM$-modules and there are decompositions
\[V\otimes_{\QQ}\CC=\oplus_{\tau \in \CT}V_\tau, \quad V^+=\oplus_{\tau\in \CT} V^+_\tau, \quad V^-=\oplus_{\tau\in \CT} V^-_\tau.\]
Set: 
\begin{equation} \label{Edefgftau}
g(\tau):=\dim_\CC V_\tau, \quad \cf(\tau):=\dim_{\CC} V^+_\tau. 
\end{equation}

For any $x\in \QQ$, let $\langle x\rangle$ denote the fractional part of $x$. 
By \cite[Lemma 2.7, \S3.2]{moonen} (or \cite{deligne-mostow}), 
\begin{equation}\label{DMeqn}
\cf(\tau_n)=\begin{cases} -1+\sum_{i=1}^N\langle\frac{-na(i)}{m}\rangle \text{ if $n\not\equiv 0 \bmod m$}\\
0\text{ if $n\equiv 0 \bmod m$}.\end{cases}
\end{equation}
The dimension $\cf(\tau_n)$ is independent of the choice of $t\in U$. 
The \emph{signature type} of the monodromy datum $(m,N,a)$ is 
\[\cf=(\cf(\tau_1), \ldots, \cf(\tau_{m-1})).\]

\begin{remark}\label{rmk_dimV}
For all $\tau\in \CT$, $\dim_\CC V^+_{\tau^*}=\dim_\CC V^-_{\tau}$ and thus $g(\tau)=\cf(\tau)+\cf(\tau^*)$. 
The value $g(\tau)$ is constant on the $\sigma$-orbit of $\tau$, 
so we sometimes write $g(\co)=g(\tau)$ for $\tau \in \co$.
\end{remark}

\subsection{Newton polygons}\label{prelim_NP} \mbox{ } \\

Let $X$ denote a $g$-dimensional abelian scheme over an algebraically closed field $\FF$ of positive characteristic $p$ (e.g., $X=\CX_x$, where $\CX$ is the universal abelian scheme on $\CA_g$, and $x\in \CA_g(\FF)$). 

If $\FF$ is an algebraic closure of $\FF_p$, the finite field of $p$ elements, then there exists a finite subfield $\FF_0\subset \FF$ such that $X$ is isomorphic to the base change to $\FF$ of an abelian scheme $X_0$ over $\FF_0$. Let $W(\FF_0)$ denote the Witt vector ring of $\FF_0$. 
Consider the action of Frobenius $\varphi$ on the crystalline cohomology group $H^1_{\rm cris}(X_0/W(\FF_0))$. There exists an integer $n$ such that $\varphi^n$, the composite of $n$ Frobenius actions, is a linear map on $H^1_{\rm cris}(X_0/W(\FF_0))$.
The {\em Newton polygon} $\nu(X)$ of $X$ is defined as the multi-set of rational numbers $\lambda$, called the {\em slopes},
such that $n\lambda$ are the valuations at $p$ of the eigenvalues of Frobenius for this action. Note that the Newton polygon is independent of the choice of $X_0$, $\FF_0$, and $n$. 
For a more general $\FF$, one can use the Dieudonn\'e--Manin classification to define the Newton polygon.

The {\em $p$-rank} of $X$ is the multiplicity of the slope $0$ in $\nu(X)$; it equals $\dim_{\FF_p}\hom(\mu_p,X)$.

If $\CA$ is an abelian variety or $p$-divisible group defined over a local field of mixed characteristic $(0,p)$, 
we write $\nu(\CA)$ for the Newton polygon of its special fiber.

For $\lambda \in \QQ \cap [0,1]$, the multiplicity $m_\lambda$ is the multiplicity of $\lambda$ 
in the multi-set; if
$c,d\in\NN$ are relatively prime integers such that $\lambda=c/(c+d)$, then $(c+d)$ divides $m_\lambda$. 
The Newton polygon is {\em symmetric} if $m_\lambda=m_{1-\lambda}$ for every $\lambda\in \QQ\cap [0,1]$. 
The Newton polygon is typically drawn as a lower convex polygon, 
with slopes equal to the values of $\lambda$ occurring with multiplicity $m_\lambda$.
The Newton polygon of a $g$-dimensional abelian variety $X$ is symmetric and, when drawn as a polygon,  
it has endpoints $(0,0)$ and $(2g,g)$ and integral break points. 

Let $ord$ denote the Newton polygon  $\{0,1\}$ and $ss$ denote the Newton polygon $\{1/2,1/2\}$.
Thus an ordinary (resp.\ supersingular) abelian variety of dimension $g$
has Newton polygon $ord^g$ (resp.\ $ss^g$).
For $s,t\in \NN$, 
with $s \leq t/2$ and ${\rm gcd}(s,t)=1$, let $(s/t, (t-s)/t)$ denote the Newton polygon with slopes $s/t$ and  $(t-s)/t$, each with multiplicity $t$. 

We denote the union of multi-sets by $\oplus$, and for any multi-set $\nu$, and $n\in\NN$, we write $\nu^n$ for $\nu\oplus \cdots \oplus \nu$, 
$n$-times. 
For convex polygons, we write $\nu_1\geq \nu_2$ if $\nu_1,\nu_2$ share the same endpoints and $\nu_1$ lies below $\nu_2$. 
Under this convention, the ordinary Newton polygon is maximal and the supersingular Newton polygon is minimal. 

\subsection{Dieudonn\'e modules modulo $p$ and Ekedahl--Oort types} \label{SnotDM}  \mbox{ } \\

The $p$-torsion $X[p]$ of $X$ is a symmetric ${\rm BT}_1$-group scheme (of rank $2g$) annihilated by $p$.
The \emph{$a$-number} of $X$ is ${\rm dim}_{\mathbb F}(\alpha_p, X[p])$, where $\alpha_p$ is the kernel of Frobenius on $\GG_a$.

Let ${\mathbb E}$ denote the non-commutative ring over $\overline{\FF}_p$,
generated by a $\sigma$-linear $F$ and $\sigma^{-1}$-linear $V$ with $FV=VF=0$.
The mod $p$ Dieudonn\'e module of $X$ is an $\EE$-module of finite dimension ($2g$).
Over $\FF$, 
there is an equivalence of categories between finite commutative group schemes of rank $2g$ annihilated by
$p$ and $\EE$-modules of dimension $2g$.  These can equivalently be classified by Ekedahl--Oort types.

Let $\EE(w)$ denote the left ideal of ${\mathbb E}$ generated by $w$.
The mod $p$ Dieudonn\'e module for an ordinary elliptic curve is isomorphic to
$L:=\EE/\EE(F, V-1) \oplus \EE/\EE(V, F-1)$; 
for a supersingular elliptic curve, it is isomorphic to $N_{1,1}:=\EE/\EE(F-V)$.

\section{Shimura varieties} \label{SShimura}

Let $(m,N,a)$ be a monodromy datum with $N \geq 4$, and $\cf$ the associated signature type given by (\ref{DMeqn}). In \cite{deligne-mostow}  Deligne and Mostow construct the smallest PEL-type Shimura variety containing $Z(m,N,a)$, which we will denote by $\Sh(\mu_m,\cf)$.
In this section, we recall the basic setting for PEL-type Shimura varieties, and the construction of \cite{deligne-mostow}.  We follow \cite{moonen}.

\subsection{Shimura datum for the moduli space of abelian varieties}\label{sec_Sh_GSp} \mbox{ } \\

Let $V=\QQ^{2g}$, 
and let $\Psi:V\times V\to\QQ$ denote the standard symplectic form.\footnote{In Section \ref{prelim_curve}, we use $V$ to denote $H^1(C_x, \QQ)$ for $x\in U(\CC)$. Here, by the convention of Deligne, $V$ is identified with $H_1(\CX_x,\QQ)$ for any $x\in \CA_g(\CC)$. The symplectic form will be identified with the one induced by the polarization on $\CX_x$. This symplectic form, along with Poincar\'e duality, induces a natural isomorphism $H^1(C_x, \QQ)(1)\cong H_1(C_x, \QQ)$. Hence the only difference between them is given by the Tate twist and we use the same notation $V$ for both.}
Let $G:={\rm GSp}(V,\Psi)$ denote the group of symplectic similitudes over ${\mathbb Q}$. 
Let $\hh$ denote the space of
homomorphisms $h:{\mathbb S}={\rm Res}_{\CC/\RR} \GG_m \to G_\RR$ which define a Hodge structure of type (-1,0)+(0,-1) on $V_\ZZ$ such that $\pm(2\pi i)\Psi$ is a polarization on $V$.
The pair $(G,\hh)$ is the Shimura datum for $\CA_g$.

Let $H\subset G$ be an algebraic subgroup over $\QQ$ such that the subspace
\[\hh_H:=\{h\in\hh\mid h \text{ factors through } H_\RR\}\] is non-empty. Then $H(\RR)$ acts on $\hh_H$ by conjugation, and for each $H(\RR)$-orbit $Y_H\subset \hh_H$, the Shimura datum $(H,Y_H)$ defines an algebraic substack $\Sh(H,Y_H)$ of $\CA_g$. 
In the following, for $h\in Y_H$, we sometimes write $(H,h)$ for the Shimura datum $(H,Y_H)$.
For convenience, we also write $\Sh(H,\hh_H)$ for the finite union of the Shimura stacks $\Sh(H,Y_H)$, as $Y_H$ varies among the $H(\RR)$-orbits in $\hh_H$.

\subsection{Shimura data of PEL-type}\label{sec_Sh_PEL} \mbox{ } \\

Now we focus on Shimura data of PEL-type. 
Let $B$ be a semisimple $\QQ$-algebra, together with an involution $*$. Suppose there is an action of $B$ on $V$ such that 
$\Psi(bv,w)=\Psi(v,b^*w)$, for all $b\in B$ and all $v,w\in V$.
Let
\[H_B:={\rm GL}_B(V)\cap {\rm GSp}(V,\Psi).\]  
We assume that $\hh_{H_B}\neq \emptyset$.

For each $H_B(\RR)$-orbit $Y_B:=Y_{H_B}\subset \hh_{H_B}$, the associated Shimura stack 
$\Sh(H_B,Y_B)$ arise as moduli spaces of polarized abelian varieties 
endowed with a $B$-action, and are called of PEL-type. In the following, we also write $\Sh(B):=\Sh(H_B,\hh_{H_B})$.

Each homomorphism $h\in Y_B$ defines a decomposition of $B_\CC$-modules \[V_\CC=V^+\oplus V^-\]
where $V^+$ (respectively, $V^-$) is the subspace of $V_\CC$ on which $h(z)$ acts by $z$ (respectively, by $\bar{z}$).
The isomorphism class of the $B_\CC$-module $V^+$ depends only on $Y_B$. Moreover, $Y_B$ is determined by the isomorphism class of $V^+$ as a $B_\CC$-submodule of $V_\CC$.
In the following, we prescribe $Y_B$ in terms of the $B_\CC$-module $V^+$. 
By construction, $\dim_\CC V^+=g$.

\subsection{Shimura subvariety attached to a monodromy datum} \label{sec_Sh_cyclic} \mbox{ } \\

We consider cyclic covers of the projective line branched at more than three points; fix a monodromy datum $(m,N,a)$ with $N\geq 4$. Take $B=\QQ[\mu_m]$ with involution $*$ as in Section~\ref{prelim_gpalg}.

As in Section \ref{prelim_curve}, let $C \to U$ denote the universal family of $\mu_m$-covers of ${\mathbb P}^1$ branched at $N$ points with inertia type $a$; let $j=j(m, N, a): U \to \CA_g$ be the composition of the Torelli map with the morphism $U \to \CM_g$.
From Definition \ref{DZmNa}, recall that $Z=Z(m,N,a)$ is the closure in $\CA_g$ of the image of $j(m,N,a)$. 

The pullback of the universal abelian scheme $\CX$ on $\CA_g$ via $j$ is the relative Jacobian $\CJ$ of $C\rightarrow U$. Since $\mu_m$ acts on $C$, there is a natural action of the group algebra $\ZZ[\mu_m]$ on $\CJ$. 
We also use $\CJ$ to denote the pullback of $\CX$ to $Z$.  The action of $\ZZ[\mu_m]$ extends naturally to $\CJ$ over $Z$.
Hence the substack $Z=Z(m,N,a)$ is contained in $\Sh(\QQ[\mu_m])$ for an appropriate choice of a structure of $\QQ[\mu_m]$-module on $V$.
More precisely, fix $x\in Z(\CC)$, and let $(\CJ_x,\theta)$ denote the corresponding Jacobian with its principal polarization $\theta$.  
Choose a symplectic similitude, meaning an isomorphism
\[\alpha: (H_1(\CJ_x,\QQ), \psi_\theta)\to (V ,\Psi),\]
such that the pull back of the symplectic form $\Psi$ to $H_1(\CJ_x, \QQ)$ 
is a scalar multiple of $\psi_\theta$,
where $\psi_\theta$ denotes the Riemannian form on $H_1(\CJ_x,\QQ)$ corresponding to the polarization $\theta$.
Via $\alpha$, the $\QQ[\mu_m]$-action on $\CJ_x$ induces an action on $V$. This action satisfies \[\hh_{\QQ[\mu_m]}\neq \emptyset, \text{ and }\Psi(bv,w)=\Psi(v,b^*w),\] for all $b\in \QQ[\mu_m]$, all $v,w\in V$,
and $Z\subset \Sh(\QQ[\mu_m])$.

The isomorphism class of $V^+$ as a $\QM\otimes_\QQ\CC$-module  is determined by and determines 
the signature type $\{\cf(\tau)=\dim V^+_\tau\}_{\tau \in \CT}$.
By \cite[2.21, 2.23]{deligne-mostow} (see also \cite[\S\S 3.2, 3.3, 4.5]{moonen}), 
the $H_\QM(\RR)$-orbit $Y_{\QQ[\mu_m]}$ in $\hh_{H_\QM}$ such that \[Z\subset \Sh(H_\QM, Y_\QM) \] corresponds to the isomorphism class of $V^+$ with $\cf$ given by \eqref{DMeqn}. From now on, since $\Sh(H_\QM, Y_\QM)$ depends only on $\mu_m$ and $\cf$, we denote it by $\Sh(\mu_m,\cf)$.

The irreducible component of $\Sh(\mu_m,\cf)$ containing $Z$ is the largest closed, reduced and irreducible substack $S$ of $\CA_g$ containing $Z$ such that the action of $\ZZ[\mu_m]$ on $\CJ$ extends to the universal abelian scheme over $S$. To emphasis the dependence on the monodromy datum, we denote this irreducible substack by $S(m,N,a)$.

\section{Newton polygons for special Shimura varieties}\label{sec_posdim}

In this section, we determine the Newton polygons and the $\mu$-ordinary Dieudonn\'e modules that occur 
for the special families of \cite[Table 1]{moonen}.
We refer to \cite[Section 3]{LMPT} for a review of the Shimura--Taniyama method for computing the Newton polygon of a cyclic cover of the projective line branched at exactly three points. 

\subsection{Newton polygon stratification for PEL-type Shimura varieties}\label{sec_VW} \mbox{ } \\

We recall some of the key results about the Newton polygon stratification for PEL-type Shimura varieties 
at unramified primes of good reduction from \cite{viehmann-wedhorn}.
Similar results are now known for abelian-type Shimura varieties by \cite{shen-zhang} and we refer to 
\cite{HeRapoport} for a survey of previous work.

Let $(H,h)$ denote a Shimura datum of PEL-type, with $H$ connected. Assume $p$ is an unramified prime of good reduction for $(H,h)$, 
then the Shimura variety $\Sh(H,h)$ has a smooth canonical integral model at $p$, and
we denote by $\Sh(H,h)_{\overline{\FF}_p}$ its special fiber (base changed to $\overline{\FF}_p$).
\footnote{More precisely, for $p$ an unramified prime for the datum $(H,h)$, the group 
$H_{\QQ_p}$ is the generic fiber of some reductive group $\mathcal{H}$ over $\ZZ_p$ with connected geometric fibers,
and we assume that the level at $p$ of the Shimura variety $\Sh(H,h)$ is the hyperspecial subgroup $\mathcal{H}(\ZZ_p)\subset H(\QQ_p)$. The condition 
for $p$ to be an unramified prime of good reduction for
$(H,h)$ and the definition of the canonical integral model of $\Sh(H,h)$ at $p$ can be found in  \cite[\S1]{viehmann-wedhorn}.}
For the Shimura data introduced in Section \ref{sec_Sh_cyclic}, $H$ is always connected and $p$ is an unramified prime of good reduction if $p \nmid m$. 

There exists a $p$-adic cocharacter $\mu=\mu_h$ which factors through $H(\overline{\QQ}_p)$ and is conjugate 
by an element in $H$ to the Hodge cocharacter induced by $h$.  
Given the local datum $(H_{\QQ_p}, \mu)$, in \cite[\S 5]{kottwitz1} 
and \cite[\S 6]{kottwitz2}, Kottwitz introduced a partially ordered set $B(H_{\QQ_p},\mu)$, and showed that $B(H_{\QQ_p},\mu)$ 
has a canonical {maximal} 
element (called the \emph{$\mu$-ordinary} element)
and a minimal one (called \emph{basic}).
For the Shimura datum $(G={\rm GSp}(V,\Psi),h)$ as in Section \ref{sec_Sh_GSp}, 
the set $B({\rm G_{\QQ_p}},\mu)$ is canonically isomorphic (as a partially ordered set) to 
the set of all symmetric convex polygons, with endpoints $(0,0)$ and $(2g,g)$, integral breakpoints, and slopes in $\QQ\cap [0,1]$
(see \cite[\S 4]{kottwitz-rapoport}). Under this identification, the $\mu$-ordinary and basic elements are respectively the ordinary and supersingular Newton polygons.

Furthermore, for a Shimura datum  $(H,h)\subseteq (G,h)$, Kottwitz constructed a canonical map (the Newton map)
\[\nu :B(H_{\QQ_p},\mu)\to B(G_{\QQ_p},\mu).\]

For any $b\in B(H_{\QQ_p},\mu)$, consider the Newton polygon stratum
\[\Sh(H,h)_{\overline{\FF}_p}(b):=\{x\in |\Sh(H,h)_{\overline{\FF}_p}| \mid \nu(\CX_x)=\nu(b)\},\]
where, as in Section \ref{prelim_NP}, $\CX$ denotes the universal abelian scheme over $\CA_g$.
Viehmann and Wedhorn prove:

\begin{theorem} \label{TNPset} \cite[Theorem 1.6]{viehmann-wedhorn} 
For any $b\in B(H_{\QQ_p},\mu)$ and any irreducible component $S$ of $\Sh(H,h)$, there exists 
$x\in |S_{\overline{\FF}_p}|$ such that $\nu(\CX_x)=\nu(b)$. 
\end{theorem}

\begin{definition} \label{DdefNPset}
Let $(m, N,a)$ be a monodromy datum with signature $\cf$.
For the associated Shimura datum $(H_{\QM},Y_\QM)$
as defined in Section \ref{sec_Sh_cyclic}, we write $B(\mu_m,\cf)$ in place of $B((H_{\QM})_{\QQ_p},\mu(\cf))$
and denote by $\nu(B(\mu_m,\cf))$ its image in $B(G_{\QQ_p},\mu)$.
\end{definition}

We note that the sets $B(\mu_m,\cf)$ and $\nu(B(\mu_m,\cf))$ depend on $p$, specifically on the congruence of $p$ modulo $m$.

By Theorem \ref{TNPset}, $\nu(B(\mu_m,\cf))$ is the set of Newton polygons appearing on 
the irreducible component $S=S(m, N,a)$. 
 We refer to
 the images under $\nu$ of the $\mu$-ordinary and basic elements of  $B(\mu_m,\cf)$ respectively as the $\mu$-ordinary and basic Newton polygons.

\subsection{The $\mu$-ordinary Newton polygon}\label{muordformula} \mbox{ } \\

For a PEL-type Shimura variety, Moonen \cite{moonenST} explicitly computes the slopes of the
$\mu$-ordinary Newton polygon at an unramified prime of good reduction in terms of the signature type.
Here, for convenience, we recall the formula in the case of  the Shimura variety $\Sh(\mu_m,\cf)$, 
following \cite[\S2.8]{eischenmantovan}.

Recall from Sections \ref{prelim_gpalg}-\ref{prelim_curve} that $\sigma$ is the Frobenius action on $\CT=\hom(\QQ[\mu_m], \CC)$ and $\CO$ is the set of $\sigma$-orbits in $\CT$. The signature type $\cf$ is given by \eqref{DMeqn}.
For $\tau \in \CT$, recall that $\cf(\tau)$ is the dimension of the eigenspace of $V^+$ for $\tau$ under the $\mu_m$-action.

For a $\sigma$-orbit $\co\in \CO$, let $s(\co)$ be the number of distinct values of 
$\{\cf(\tau)\mid\tau\in\co\}$ in $[1,g(\co)-1]$.
Let $E(1), \ldots, E(s(\co))$ denote these distinct values,
ordered such that 
\[g(\co) > E(1)> E(2)>\cdots > E(s(\co))>0.\] 
For convenience, we also write $E(0):=g(\co)$ and $E(s(\co)+1):=0$.
 
\begin{proposition}
The $\mu$-ordinary Newton polygon for $\Sh(H,h)$ is given by $\nu_o=\oplus_{\co\in \CO}\mu(\co)$
where the polygon $\mu(\co)$ has exactly $s(\co)+1$ distinct slopes, 
\[0\leq \lambda(0) < \lambda(1) <\cdots <\lambda(s(\co))\leq 1,\]
and, for each integer $t$ such that $0\leq t\leq s(\co)$, the slope
\begin{equation}\label{slope}
\lambda(t)=\frac{1}{\#\co}\sum_{i=0}^{t} \#\{\tau\in \co \mid\cf(\tau)=E(i)\}
\end{equation}
occurs in $\mu(\co)$ with multiplicity
\begin{equation}\label{multiplicity}
m(\lambda(t))=(\#\co)\cdot(E(t)-E(t+1)).
\end{equation}
\end{proposition}

\begin{example} \label{Esiglen1}
If $\#\co =1$ with signature $(f_1)$, then $\mu(\co)$ has two slopes: $0$ with multiplicity $g(\co)-f_1$ and $1$ with multiplicity $f_1$.
\end{example}

\begin{example} \label{Esiglen2}
If $\co =\{\tau, \tau^*\}$ with signature $(f_1,f_2)$ and $f_1 \leq f_2$, then $\mu(\co)=ord^{2f_1} \oplus ss^{f_2-f_1}$.
\end{example}

\begin{proof}
If $f_1 \not = 0$ and $f_1 \not = f_2$, then $s(\co)=2$; also $E(0)=f_1+f_2$, $E(1)=f_2$, and $E(2)=f_1$.
Then $\lambda(0)=0$ with $m(0)=2f_1$; $\lambda(1)=1/2$ with $m(1/2)=2(f_2-f_1)$; and $\lambda(2)=1$ with $m(1)=2f_1$.
If $f_1=0$ or $f_1=f_2$, then a similar calculation applies.
In all cases, $\nu_o=ord^{2f_1} \oplus ss^{f_2-f_1}$.
\end{proof}

A similar calculation applies for orbits with $\#\co=2$ that are not self-dual, 
again showing that $\mu(\co)$ has slopes in $\{0, 1/2, 1\}$. 
Hence, we focus on the orbits of length $>2$ when computing the $\mu$-ordinary Newton polygon.

\subsection{The Kottwitz method}\label{sec_Kottwitz} \mbox{ } \\

Recall that $S=S(m,N,a)$ is an irreducible component of the smallest Shimura subvariety of $\CA_g$ containing $Z(m,N,a)$.
In \cite[Table 1]{moonen}, Moonen lists the 20 monodromy data $(m,N,a)$ (up to equivalence) for which $Z(m,N,a)=S(m,N,a)$.  
He obtains the list by comparing the 
dimensions of $Z$ and $S$:
on one hand, $\dim(Z)=N-3$ since $3$ out of the $N$ branch points can be fixed via fractional linear transformations;
on the other hand,
the signature type $\cf$ determines $\dim(S)$ by \cite[Equation (3.3.1)]{moonen}. 

As above, $p$ is a rational prime such that $p\nmid m$. We now give an explicit description of the set
$\nu(B(\mu_m,\cf))$ of Newton polygons from Definition \ref{DdefNPset}.
Let $M$ be the (contravariant) Dieudonn\'e module of an abelian variety $X=\CX_x$ for any $x\in\Sh(H_{\QM},\mu)(\overline{\FF}_p)$. 
The $\QM$-action on $X$ induces a $\QM\otimes_\QQ \QQ_p$-action on $X[p^\infty]$ and thus a canonical decomposition
\[X[p^\infty]=\bigoplus_{\co \in \CO} X[\p_\co^\infty].\]
There is an analogous decomposition of Dieudonn\'e modules
\[M=\bigoplus_{\co\in \CO} M_\co.\] 
Let $\nu(\co)$ denote the Newton polygon of $X[\p_\co^\infty]$.
By \eqref{DMeqn}, for $\co=\{\tau_0\}$, we have $X[\p_\co^\infty]=0$ and hence $M_\co=0$. For convenience, we set $\nu(\{\tau_0\})=0$.

\begin{theorem}[special case of {\cite[Theorem 1.6]{viehmann-wedhorn}}] \label{thm_VW} 
The Newton polygons that occur on $S(m,N,a)_{\overline{\FF}_p}$ are the Newton polygons $\oplus_{\co\in \CO} \nu(\co)$ 
satisfying conditions (D), (WA), and (M) below.
\end{theorem}

The polarization on $X$ induces a prime-to-$p$ isogeny $X\rightarrow X^\vee$ and the Rosati involution acts on $\QM$ via the involution $*$.
Thus there is an isomorphism $M\cong M^\vee(1)$ compatible with the decomposition and this isomorphism induces certain restrictions on $\nu(\co)$ in the following sense:

\emph{condition (D)}: 
\begin{enumerate}
\item if $\co^*\neq \co$, then $M_\co\cong M_{\co^*}^\vee(1)$, i.e., $\nu(\co)$ determines $\nu(\co^*)$. More precisely, if $\nu(\co)$ has slopes $\lambda_1,\dots, \lambda_s$ with multiplicities $m_1,\dots, m_s$, then $\nu(\co^*)$ has slopes $1-\lambda_1,\dots, 1-\lambda_s$ with multiplicities $m_1,\dots, m_s$.
\item if $\co^*=\co$, then $M_\co\cong M_\co^\vee(1)$, i.e., $\nu(\co)$ is symmetric. In other words, if $\lambda$ is a slope of $\nu(\co)$ of multiplicity $m_\lambda$, then $1-\lambda$ is a slope of $\nu(\co)$ of multiplicity $m_\lambda$.
\end{enumerate}

By \cite[Theorem 4.2]{rapoport-richartz}, for any $\co\in \CO$, the Newton polygon $\nu(\co)$ satisfies
the weakly admissible condition:
\begin{center}
\emph{condition (WA)}: $\nu(\co)\geq \mu(\co)$.
\end{center}

Recall that $\p_{\co}$ is a prime of $K_{d_\co}$ with inertia degree $\# \co$. Note that $K_{d_\co, \p_\co}$ acts on $M_\co$, which yields:
\begin{center}
\emph{condition (M)}: the multiplicities of slopes in $\nu(\co)$ are divisible by $\# \co$.
\end{center}

\subsection{Computation} \label{ScomputeNP} \mbox{ } \\

We compute all possible Newton polygons for the $20$ families from \cite[Table 1]{moonen};
the data can be found in the tables in \S\ref{sec_table}.  Here, the label $M[r]$ denotes the $r$th label in  \cite[Table 1]{moonen}.
Since the computations are similar for all $20$ cases, we give complete details 
for two examples:
first, family $M[17]$ where $m=7$ is prime; and second, family $M[19]$ where $m=9$ is not a prime. 
 
Set $\CO'=\CO-\{\{\tau_0\}\}$; from now on, $\co\in \CO'$.
The Newton polygon for $X$ is given by $\oplus_{\co\in \CO'} \nu(\co)$.
In the following two examples, since non-symmetric Newton polygons arise, we use 
different notation from \S\ref{prelim_NP}: when ${\rm gcd}(a,b)=1$, with small abuse of notation,
we use $G_{a,b}$ instead of $\nu(G_{a,b})$
to denote the Newton polygon of pure slope $a/(a+b)$ with multiplicity $a+b$.

\begin{example} \label{M17muordNP}
We consider family $M[17]$ where $m=7$, $N=4$, and $a=(2,4,4,4)$. 
The signature type is $(1,2,0,2,0,1)$.
Since $N=4$, the Shimura variety is $1$-dimensional. 
For each congruence class of $p \not = 7$, there are exactly two possible Newton polygons. 

\begin{enumerate}





 
\item If $p \equiv 2, 4 \bmod 7$, then $\#\co=3$ and $\#\CO=2$. In this case, $\CO=\{\co_1,\co_2\}$, where $\co_1=\{\tau_1,\tau_2, \tau_4\}$ and $\co_2=\{\tau_3,\tau_5,\tau_6\}$. 
Since $\co_1^*=\co_2$, 
by (D), $M_{\co_1}\cong M_{\co_2}^\vee(1)$. 
By Section \ref{muordformula}, $\mu(\co_2)=G_{0,1}^3\oplus G_{1,2}$. By (WA) and (M), if $\nu(\co_2)\neq \mu(\co_2)$, then $\nu(\co_2)=G_{1,5}$. Therefore, by (D),

$$\nu_o = G_{0,1}^3 \oplus G_{1,2} \oplus G_{2,1} \oplus G_{1,0}^3,  \ {\rm and} \  \nu_b = G_{1,5} \oplus G_{5,1}.$$

\item If $p \equiv 3,5 \bmod 7$, then $\#\co=6$ and $\#\CO=1$ and there is only one orbit $\co=\CT$. 

By Section \ref{muordformula},
$$\nu_o=\mu(\co) = G_{1,2}^2 \oplus G_{2,1}^2.$$ By (M), every eigenspace of the Frobenius action on 
the Dieudonn\'e module $M$ 
has dimension divisible by $6$. Since the interval $(1/3,1/2)$ contains no rational number with denominator $6$, the basic Newton Polygon
$\nu_b$ is supersingular.

\end{enumerate}

\end{example}

\begin{example}
We compute all possible Newton polygons for family $M[19]$ where $m=9$, $N=4$ and $a=(3,5,5,5)$. 
The signature is $(1,2,0,2,0,1,0,1)$. 
Then $\tau_3,\tau_6$ have order $3$ and $\tau_i$ have order $6$ for $i=1,2,4,5,7,8$. 
By \eqref{DMeqn} and Remark \ref{rmk_dimV}, $g(\tau_3)=g(\tau_6)=0+1=1$ and $g(\tau_i)=2$ for $i=1,2,4,5,7,8$. 

To illustrate the idea, we show the computation for $p\equiv 4,7 \bmod 9$. In this case, $\# \CO'=4$, with orbits $\co_1=\{\tau_1,\tau_4,\tau_7\}, \co_2=\{\tau_2,\tau_5,\tau_8\}, \co_3=\{\tau_3\}, \co_4=\{\tau_6\}$.

By Section \ref{muordformula},
$\mu(\co_3)=G_{0,1}$ and by (WA), this is the only possibility for the Newton polygon of $M_{\co_3}$. By (D), the only possible Newton polygon on $M_{\co_3}\oplus M_{\co_4}$ is $G_{0,1}\oplus G_{1,0}$. Again by Section \ref{muordformula}, $\mu(\co_1)=G_{1,2}\oplus G_{2,1}$ on $M_{\co_1}$. By (M), every slope in the Newton polygon of $M_{\co_3}$ has multiplicity $3$ or $6$ and hence the denominator divides $6$. By (WA), the first slope lies in the interval $[1/3,1/2]$.
So, there are two possible Newton polygons 
\[
\nu_o(\co_1)=G_{1,2}\oplus G_{2,1},  \ {\rm and} \  \nu_{b}(\co_1)=G_{1,1}^3.
\]
Then, by (D), the only possible Newton polygons on $M_{\co_1}\oplus M_{\co_2}$ are 
\[
\nu_{o}(\co_1,\co_2)=G_{1,2}^2\oplus G_{2,1}^2, \ {\rm and} \  \nu_{b}(\co_1,\co_2)=G_{1,1}^6.
\]
In conclusion, the two possible Newton polygons of the Dieudonn\'e module $M$ are:
\[
\nu_{o}=G_{0,1}\oplus G_{1,2}^2\oplus G_{2,1}^2\oplus G_{1,0},  \ {\rm and} \  \nu_{b}=G_{0,1}\oplus G_{1,1}^6\oplus G_{1,0}.
\]

\end{example}

\subsection{The $\mu$-ordinary Ekedahl--Oort type}\label{EO}\mbox{}\\

The Ekedahl--Oort type is a combinatorial invariant classifying the structure of the reduction modulo $p$ of the Dieudonn\'e module.
In  \cite[Theorem 1.3.7]{moonenST}, 
Moonen shows that the largest Ekedahl--Oort stratum coincides with the $\mu$-ordinary Newton stratum.  
We recall the $\EE$-module structure of the mod $p$ Dieudonn\'e module for the $\mu$-ordinary locus, given in \cite[\S 1.2.3]{moonenST}. 

Fix $\co\in\CO$ and  $t\in\{0,\dots, s(\co)\}$. 
We define the $\EE$-module: 
\[N_t(\co):=\oplus_{\tau\in\co}\bfp e_\tau,\] by 
\begin{eqnarray*}
F(e_\tau) & = & e_{\tau^\sigma} \ {\rm if } \ \cf(\tau)\leq \E(t+1), \ 0 \ {\rm otherwise},\\
V(e_{\tau^{\sigma}}) & =& e_\tau \ {\rm  if} \  \cf(\tau)>\E(t+1), \ 0 \ {\rm otherwise}.
\end{eqnarray*}

Then $N_t(\co)$ arises as the mod $p$-reduction of an isoclinic Dieudonn\'e module of slope $\lambda(t)$.  
 
The reduction modulo $p$ of the $\mu$-ordinary Dieudonn\'e module of ${\rm Sh}(H,h)_\bfp$ is 
\begin{equation} \label{Nomerge}
N_o=\bigoplus_{\co\in\CO} N(\co),\text{ where } N(\co)=\bigoplus_{t=0}^{s(\co)} N_t(\co)^{E(t)-E(t+1)}.
\end{equation}

This agrees with the formula in \cite[\S 1.2.3]{moonenST}:
note that Moonen defines an $\EE$-module for each $j\in \{1, \ldots, g(\co)\}$;
if $g(\co)-E(t+1)< j \leq g(\co)-E(t)$, then the $\EE$-modules have the same structure, 
which determines the multiplicity of each factor $N_t(\co)$ in \eqref{Nomerge}.

Recall from Section \ref{SnotDM} that 
$L$ denotes the $\EE$-module of an ordinary elliptic curve, and $N_{1,1}$ that of a supersingular elliptic curve.

\begin{example} \label{Edieumod2}
If $\co =\{\tau,\tau^*=\tau^\sigma\}$ with signature $(f_1,f_2)$ with $f_1 \leq f_2$, 
then $N(\co) \simeq L^{2f_1} \oplus N_{1,1}^{f_2-f_1}$.
\end{example}

Note that Example \ref{Edieumod2} is not automatic from Example \ref{Esiglen2} since several Dieudonn\'e modules can occur 
with Newton polygon $ss^{f_2-f_1}$; 
the one in Example \ref{Edieumod2} has $a$-number $f_2-f_1$; it
decomposes into dimension $2$ subspaces stable under $F$ and $V$.

\begin{proof}
If $f_1 \not = 0$ and $f_1 \not = f_2$, then by the computation in Example \ref{Esiglen2} and \eqref{Nomerge}
 $N_0(\co)$ is generated by $e_\tau,e_{\tau^\sigma}$ with $F={\rm id}$, the identity map, and $V=0$; $N_2(\co)$ is generated by $e_\tau,e_{\tau^\sigma}$ with $V={\rm id}$ and $F=0$; and $N_1(\co)$ is generated by $e_\tau,e_{\tau^\sigma}$ with  $F(e_\tau)=V(e_\tau) = e_{\tau^\sigma}$ and 
$F(e_{\tau^\sigma}) = V(e_{\tau^\sigma}) = 0$.
A similar calculation works when $f_1=f_2$ or $f_1=0$.
In all cases, $N(\co) \simeq L^{2f_1} \oplus N_{1,1}^{f_2-f_1}$.
\end{proof}

A similar calculation applies for all orbits with $\#\co =2$.
Hence, we focus on orbits of length $>2$ when computing the $\mu$-ordinary mod $p$ Dieudonn\'e module.

\begin{example} \label{EOM17}
We compute the $\mu$-ordinary mod $p$ Dieudonn\'e module $N_o$ for family $M[17]$ where $m=7$. 
The signature is $(1,2,0,2,0,1)$ and $g(\tau)=2$.
\begin{enumerate}


\item If $p \equiv 2,4 \bmod 7$, then the $p$-rank is $3$, and
the local-local 
part of $N_o$ has 6 generators with action of $F$ and $V$ given by:
\[\begin{array}{|c|c|c|c|}
\hline
t=0, E(t+1)=1 & e_1 & e_2 & e_4 \\
\hline
F & e_2 & 0 & 0 \\
\hline
V & e_4 & 0 & e_2 \\
\hline
\end{array}
\text{ and }
\begin{array}{|c|c|c|c|}
\hline
t=1, E(t+1)=0 & e'_3 & e'_5 & e'_6\\
\hline
F & e'_5 & e'_6 & 0 \\
\hline
V &  e'_6 & 0 & 0\\
\hline
\end{array}
\]
Thus $N_o \simeq L^3 \oplus (\EE/\EE(F-V^2))e_1 \oplus (\EE/\EE(F^2-V))e'_3$.
One can compute that $N_o$ has Ekedahl-Oort type $[1,2,3,3,4,4]$.

\item If $p \equiv 3,5 \bmod 7$, then $N_o$ has 12 generators
with the actions of $F$ and $V$ given by:

\begin{small}
\[\begin{array}{|c|c|c|c|c|c|c|}
\hline
t=0, E(t+1)=1 & e_1 & e_3 & e_2 & e_6 & e_4 & e_5\\
\hline
F &e_3 & e_2 & 0 & e_4& 0 & e_1\\
\hline
V & 0 & 0 & 0 & e_2 & 0 & e_4\\
\hline
\end{array}
\text{ and }
\begin{array}{|c|c|c|c|c|c|c|}
\hline
t=1, E(t+1)=0 & e'_1 & e'_3 & e'_2 & e'_6 & e'_4 & e'_5\\
\hline
F &0 & e'_2 & 0 & 0 & 0 & e'_1\\
\hline
V & 0 & e'_1 & 0 & e'_2 & e'_6 & e'_4\\
\hline
\end{array}\]

\end{small}
Thus 
$N_o = \EE\langle e_5,e_6 \rangle/\EE(F^3 e_5 -V e_6, V e_5 - F e_6) 
\oplus \EE\langle e'_3, e'_5 \rangle/\EE(Fe_3'- V^3 e_5', Ve_3'-Fe_5')$.
In total, $N_o$ has $4$ generators, thus $a$-number $4$.
One can compute that $N_o$ has Ekedahl-Oort type $[0,1,2,2,2,2]$.
\end{enumerate}
\end{example}

\begin{example} \label{EOM19}
We compute the $\mu$-ordinary mod $p$ Dieudonn\'e module $N_o$ for family $M[19]$ where $m=9$.  
The signature is $(1,2,0,2,0,1,0,1)$ and $g(\tau_1)=2$ and $g(\tau_3)=1$.
\begin{enumerate}
\item If $p \equiv 2, 5 \bmod 9$, there are two orbits $\co_3$ and $\co_1$. 
By Example \ref{Edieumod2}, $N(\co_3) \simeq (\EE/\EE(F-V)) e_3$.
For the orbit $\co_1$ of $\tau_1$, we see that
$N(\co_1)$
has $12$ generators with the actions of $F$ and $V$ given by:

\begin{small}
\[\begin{array}{|c|c|c|c|c|c|c|}
\hline
t=0, E(t+1)=1 & e_1 & e_2 & e_4 & e_8 & e_7 & e_5\\
\hline
F & e_2 & 0 & 0 & e_7 & e_5 & e_1\\
\hline 
V & 0 & 0 & e_2 & e_4 & 0 & 0 \\
\hline
\end{array}
\text{ and }
\begin{array}{|c|c|c|c|c|c|c|}
\hline
t=1, E(t+1)=0 & e'_1 & e'_2 & e'_4 & e'_8 & e'_7 & e'_5\\
\hline
F & 0 & 0 & 0 & 0 & e'_5 & e'_1\\
\hline 
V & 0 & e'_1 & e'_2 & e'_4 & e'_8 & 0 \\
\hline
\end{array}\]
\end{small}

Then $N(\co_1) = (\EE/\EE(F^4-V^2))e_8 \oplus (\EE/\EE(V^4-F^2))e'_7$.
The Ekedahl-Oort type of $N(\co_1)$ is $[0,1,2,2,3,4]$.
In total, $N_o$ has 3 generators, thus $a$-number $3$.

\item If $p \equiv 4,7 \bmod 9$, there are four orbits.  We check that:
\begin{eqnarray*}
N(\co_3) \oplus N(\co_6) & \simeq & L\\
N(\co_1) & \simeq & (\EE/\EE(F^2-V))e_7 \oplus (\EE/\EE(F-V^2)) e'_7\\
N(\co_2) & \simeq & (\EE/\EE(F^2-V))e_8 \oplus (\EE/\EE(F-V^2)) e'_5.
\end{eqnarray*}

\end{enumerate}

One interesting feature is that the Newton polygon $(1/3, 2/3)^2$ matches with two different Dieudonn\'e modules in
Examples \ref{EOM17}(2) and \ref{EOM19}(2).
\end{example}

\section{Newton polygons of smooth curves} \label{sec_smooth}

Let $(m,N,a)$ be a monodromy datum.  If it is the monodromy datum of one of 
the twenty special families from \cite[Table~1]{moonen}, we call it \emph{special}.
Fix a prime $p$ with $p \nmid m$.

For a Newton polygon $\nu \in \nu(B(\mu_m,\cf))$,
by definition $\nu$ occurs as the Newton polygon of an abelian variety represented by a point of $Z(m,N,a)\subset\CA_g$.
This abelian variety is the Jacobian of a curve \emph{of compact type}.
One can ask whether $\nu$ occurs as the Newton polygon of the Jacobian of a curve that is \emph{smooth}.
By Definition~\ref{DZmNa}, it is equivalent to ask whether $\nu$ occurs as the Newton polygon for a point of $Z^0(m,N,a)$. 
In general, this is a subtle and difficult question. 
In this section, we provide an answer in three settings: when the Newton polygon $\nu$ is $\mu$-ordinary in Section~\ref{Smuord},
PEL-indecomposable in Section~\ref{SPELind}, or basic in Section~\ref{Ssmbasic}.  
In Section~\ref{conclusion}, we use these three criteria and the purity theorem
to prove that all of the Newton polygons for the special families in \cite[Table 1]{moonen} occur for Jacobians of smooth curves
except for a few of the supersingular ones. 

When there is no risk of confusion, we write $Z$ and $Z^0$ for $Z(m,N,a)$ and $Z^0(m,N,a)$ respectively.

\subsection{The $\mu$-ordinary case} \label{Smuord} \mbox{}\\

\begin{proposition} \label{muord}
For any special monodromy datum $(m,N,a)$, and any prime $p$ with $p \nmid m$, 
the $\mu$-ordinary Newton polygon $\nu_o\in \nu(B(\mu_m,\cf))$  
and the $\mu$-ordinary mod $p$ Dieudonn\'e module $N_o$ occur for the Jacobian of a smooth curve in $Z(m,N,a)$. 
\end{proposition}
\begin{proof}
Both $Z^0$ and the $\mu$-ordinary Newton polygon
(or Ekedahl--Oort) stratum $Z(\nu_o)$ are open and dense in $Z$, thus their intersection is non-empty.
\end{proof}

\subsection{PEL-indecomposable Newton polygons} \label{SPELind} \mbox{}\\

Suppose $C$ is a curve of genus $g$ of compact type.
If the Newton polygon of ${\rm Jac}(C)$ is indecomposable as a symmetric Newton polygon of height $2g$,
(e.g., $G_{1,g-1}\oplus G_{g-1,1}$), then $C$ is necessarily smooth. 
In this section, we refine that observation for curves of compact type which are $\mu_m$-Galois covers of the projective line.

\begin{remark} \label{Raddmiss}
The following material is based on information about the boundary of the Hurwitz space of 
cyclic covers of the projective line, see e.g., \cite[Chapter 4]{wewersthesis} and \cite{Ekedahl94}. 
A point $\xi$ of $Z(m,N,a)-Z^0(m,N,a)$ represents the Jacobian of a singular curve $Y$.  
This curve $Y$ is of compact type since ${\rm Jac}(Y)$ is an abelian variety.
Furthermore, there is a $\mu_m$-cover $\psi: Y \to X$ where $X$ is a singular curve of genus $0$.
The cover $\psi$ is the join of two $\mu_m$-covers $\psi_1:Y_1 \to X_1$ and $\psi_2:Y_2 \to X_2$, 
where the curves are clutched together in ordinary double points.
Since $\xi$ is in the closure of $Z^0(m,N,a)$, the cover $\psi$ is \emph{admissible}, meaning that
at each clutching point the canonical generators of inertia for $\psi_1$ and $\psi_2$ are inverses.
Since $Y$ is connected, either $Y_1$ or $Y_2$ is connected;
without loss of generality, say $Y_1$ is connected and let $r$ be the number of components of $Y_2$.
Then $\psi_2$ is induced from a $\mu_{m/r}$-cover $\psi_2^{\rm co}:Y_2^{\rm co} \to X_2$, where $Y_2^{\rm co}$ 
is isomorphic to a connected component of $Y_2$.
Finally, the fact that $Y$ is of compact type implies that each component of $Y_2$ is 
clutched together with $Y_1$ at exactly one point.
\end{remark}

In the situation of Remark \ref{Raddmiss}, the monodromy data for the covers satisfy certain 
numerical conditions as follows.

\begin{notation} \label{Nadddeg}
Let $(m, N, a)$ be the monodromy datum for $\psi$.
Let $(m, N_1, \alpha_1)$ be the monodromy datum for $\psi_1$.
Let $(m/r, N_2, \alpha_2)$ be the monodromy datum for $\psi_2^{\rm co}$.
Write \[\alpha_1=(\alpha_1(1), \ldots, \alpha_1(N_1)), \ \alpha_2=(\alpha_2(1), \ldots, \alpha_2(N_2)).\]
Then $N_1+N_2 = N+2$.
Let $\tilde{\alpha}_2 = {\rm Ind}_{m/r}^{m} \alpha_2 := (r\alpha_2(1), \ldots, r\alpha_2(N_2))$; 
we call it the monodromy datum for $\psi_2$, even if $r \not = 1$. 
The admissible condition is that $\alpha_1(N_1) \equiv - r \alpha_2(1) \bmod m$.
Then (possibly after rearranging) 
\[a = (\alpha_1(1), \ldots, \alpha_1(N_1-1), r\alpha_2(2), \ldots, r\alpha_2(N_2)).\] 
Finally, $r={\rm gcd}(m, \alpha_1(N_1))$.
\end{notation}

\begin{definition} \label{Ddegct}
We say that the pair $\alpha_1,\tilde{\alpha}_2$ is a degeneration of compact type of the inertia type $a$
if the numerical conditions in Notation~\ref{Nadddeg} are satisfied.
\end{definition}

Suppose that $\alpha_1,\tilde{\alpha}_2$ is a degeneration of compact type of $a$.
For $i=1,2$, write $Z_1:=Z(m,N_1,\alpha_1)\subset \CA_{g_1}$ and $Z_2 :=Z(m/r, N_2, \alpha_2)\subset \CA_{g_2}$.
The numerical conditions imply that $g_1+rg_2=g$.
We define $\oplus:\CA_{g_1}\times \CA_{g_2}\to\CA_g$, by $(A_1,A_2)\mapsto A_1\oplus A_2^r$.
Let $\cf$ (resp.\ $\cf_1$, $\cf_2$) be the signature for $a$ (resp.\ $\alpha_1$, $\alpha_2$).
Consider ${\rm Sh}(\mu_m,\cf_1)$ (resp.\ ${\rm Sh}(\mu_{m/r},\cf_2)$),
the smallest PEL-type Shimura stack containing $Z_1$ (resp.\ $Z_2$) 
as introduced in Section \ref{sec_Sh_cyclic}. 
It follows from the definitions that 
\[\oplus \left({\rm Sh}(\mu_m,\cf_1) \times {\rm Sh}(\mu_{m/r},\cf_2)\right)\subseteq {\rm Sh}(\mu_m,\cf).\]

\begin{definition} \label{Dpeldec}
Let $(m, N, a)$ be a monodromy datum and let $\cf$ be the corresponding signature.
A Newton polygon $\nu\in \nu(B(\mu_m,\cf))$ is {\em PEL-decomposable} if 
$\nu = \nu_1 \oplus \nu_2^r$ for some 
$\nu_1 \in \nu(B(\mu_m,\cf_1))$ and $\nu_2 \in \nu(B(\mu_{m/r},\cf_2))$, 
for some pair $\cf_1, \cf_2$ of signatures arising 
from a degeneration of compact type of $a$.
\end{definition}

We note that the condition of being PEL-decomposable depends on $p$ modulo $m$.

\begin{proposition}\label{dec}\label{indecomp}
For a special monodromy datum $(m,N,a)$, and a prime $p$ with $p \nmid m$, if $\nu\in \nu(B(\mu_m,\cf))$ 
is not PEL-decomposable,
then $\nu$ occurs as the Newton polygon of the Jacobian of a smooth curve in $Z(m,N,a)$. 
\end{proposition}

\begin{proof}
For a special monodromy datum, every Newton polygon $\nu\in \nu(B(\mu_m,\cf))$ 
occurs for a point of $Z(m,N,a)$.
By Remark \ref{Raddmiss}, if this point is the Jacobian of a singular curve, then 
$\nu$ is PEL-decomposable by Notation \ref{Nadddeg} and Definitions \ref{Ddegct} and \ref{Dpeldec}.
\end{proof}

We compute all PEL-decomposable Newton polygons, for families $M[1-20]$ 
from \cite[Table~1]{moonen}; the data for families $M[3-20]$
can be found in Lemma \ref{degen}.  
The computations are similar in all cases, so we give complete details only for $M[15]$.

\begin{example} \label{Em15add}
	For family $M[15]$ where $m=8$, $N=4$, and $a=(2,4,5,5)$; the signature is $\cf=(1,1,0,0,2,0,1)$.  
	There are two degenerations of compact type of $a$: 
	\[\alpha_1=(2,5,1), \ \alpha_2=(7,4,5) \ {\rm or} \  \alpha'_1=(5,5,6), \ \tilde{\alpha}'_2={\rm Ind}_4^8(1,2,1).\]
	The corresponding partitions of the signature $\cf$ are respectively:
	\begin{eqnarray*}
	\cf_1 & = & (1,1,0,0,1,0,0), \ \cf_2=(0,0,0,0,1,0,1);\\
	\cf'_1 & = & (0,1,0,0,1,0,1), \ \tilde{\cf}'_2=(1,0,0,0,1,0,0).
	\end{eqnarray*}
	Thus, for each non-trivial congruence class of $p \bmod  8$, 
	there are at most two PEL-decomposable Newton polygons $\eta,\eta'$. 
	We compute $\eta,\eta'$ using the Kottwitz method. 
	\begin{enumerate}
		\item If $p\equiv 1 \bmod 8$: $\eta=\eta'=\nu_o=G_{0,1}^5\oplus G_{1,0}^5$.
		
		\item If $p\equiv 3 \bmod 8$: $\eta=\eta'=\nu_o=G_{0,1}^2 \oplus G_{1,1}^3 \oplus G_{1,0}^2$.
		
		\item If $p\equiv 5 \bmod 8$: $\eta=\eta'=\nu_o=G_{0,1}^3 \oplus G_{1,1}^2 \oplus G_{1,0}^3$.
		
		\item If $p\equiv 7 \bmod 8$: $\eta=\eta'=\nu_b=G_{1,1}^5$.	
	\end{enumerate}
	
	We deduce that the following Newton polygons are not PEL-decomposable and thus occur as
	the Newton polygon of the Jacobian of a smooth curve in $Z^0(8,4,a)$: 
	\begin{enumerate}
		\item for $p\equiv 1 \bmod 8$, $\nu_b=G_{0,1}^3\oplus G_{1,1}^2\oplus G_{1,0}^3$; 
		\item for $p\equiv 3 \bmod 8$, $\nu_b=G_{1,3}\oplus G_{3,1}\oplus G_{1,1}$;
		\item for $p\equiv 5 \bmod 8$, $\nu_b= G_{0,1} \oplus G_{1,3}\oplus G_{3,1}\oplus G_{1,0}$.
		\item for $p\equiv 7\bmod 8$,  $\nu_o=G_{0,1}^2 \oplus G_{1,1}^3 \oplus G_{1,0}^2$.
	\end{enumerate}
	Thus, combining Propositions \ref{muord} and \ref{indecomp}, in this example we conclude that all Newton polygons  $\nu\in B(\mu_m,\cf)$ occur on $Z^0(8,4,a)$,  thus occur  as
	Newton polygons of the Jacobian of a smooth curve, except for $p\equiv 7\bmod 8$ and $\nu=G_{1,1}^5$. This last case is 
addressed by Proposition \ref{basic}.
	\end{example}

\subsection{Basic Newton polygons} \label{Ssmbasic} \mbox{}\\

Consider a special monodromy datum $(m,N,a)$, and a prime $p$ with $p \nmid m$, such that the set $\nu(B(\mu_m,\cf))$ 
contains exactly two Newton polygons, namely the $\mu$-ordinary polygon $\nu_o$ and the basic polygon $\nu_b$.  
In other words, assume that the closed (basic) Newton stratum has codimension $1$.
Note that this condition depends on $p$, as seen in the tables. 
For the special families of \cite[Table~1]{moonen}, this condition implies that the dimension of $Z=Z(m,N,a)$ is either $1$ or $2$, or equivalently $N=4,5$.  

Under this condition, we prove that the basic polygon $\nu_b$ occurs as the Newton polygon of a smooth curve in $Z$.

\begin{proposition}\label{basic}
For any special monodromy datum $(m,N,a)$ with $N=4$, and any sufficiently large prime $p$ with $p\nmid m$, the basic polygon $\nu_b\in
\nu(B(\mu_m,\cf))$ 
occurs as the Newton polygon of the Jacobian of a smooth curve in $Z(m,N,a)$. 
\end{proposition}

The key input of the proof is that the number of irreducible components of the basic locus in the $\bmod p$ reduction of a 
simple unitary Shimura variety is unbounded as $p$ goes to $\infty$. 
We prove this statement in Theorem \ref{prop_irredcomp} in the appendix.

\begin {proof}
For $N=4$ (dimension $1$), the set $\nu(B(\mu_m,\cf))$ contains exactly two Newton polygons for any prime $p\nmid m$.
We argue by contradiction; let us assume that the basic Newton stratum $Z_p(\nu_b)$ is contained in the mod $p$ fiber $\partial Z_p$ of the boundary $\partial Z:=Z-Z^0$.
Since the closed substack $Z_p(\nu_b)$ has codimension 1 in $Z_p$, if $ Z_p(\nu_b)\subset \partial Z_p$, then it is equal to a union of irreducible components of $\partial Z_p$, which represent the singular curves in the family. 
On one hand, the number of irreducible components of the mod $p$ fiber $\partial Z_p$ has an natural upper bound which is independent of $p$. 

 On the other hand, the number of irreducible components 
of the basic Newton stratum of the $\bmod p$ reduction of any irreducible component of the unitary Shimura variety attached to the special monodromy datum $(m, 4, a)$ grows to infinity with $p$ (see appendix Theorem \ref{prop_irredcomp}, Corollary \ref{cor_irredcomp}).  Hence the contradiction, for $p$ sufficiently large.
\end{proof}

\begin{example}
For family $M[17]$, the Shimura variety $S(7,4, (2,4,4,4))$ has dimension $1$. 
We deduce that for $p$ sufficiently large, and $p\equiv 3,5,6 \bmod 7$, the basic Newton polygon $\nu_b=G_{1,1}^6$ occurs as the Newton polygon of the Jacobian of a smooth curve in $Z$. 
\end{example}

\begin{remark}
When $N=5$ (dimension $2$), to apply Theorem \ref{prop_irredcomp} 
we have to further assume $p$ is split in the quadratic imaginary extension.  
This condition excludes primes $p\equiv -1\bmod m$ for the special families $M[6]$, $M[8]$, and $M[14]$, 
and primes $p\not\equiv 1\bmod 5$ for $M[16]$. 
\end{remark}

\subsection{Conclusion}\label{conclusion}\mbox{}\\

We apply the three criteria above to the special families in \cite[Table 1]{moonen} 
to prove that all but a few of the Newton polygons occur for the Jacobian of a smooth curve in the family.

Let $M[r]$ denote one of the special families from \cite[Table 1]{moonen} and $(m,N,a)$ its monodromy datum.
Section \ref{sec_table} contains the list of all the Newton polygons $\nu \in \nu(B(\mu_m, \cf))$.
Note that $\dim(Z(m,N,a))$ equals $3$ for the family $M[r]$ when $r=2,10$, 
equals $2$ for $r= 6, 8, 14, 16$, and equals $1$ otherwise.  

\begin{theorem} \label{onemore}
Let $(m,N,a)$ denote the monodromy datum for one of the special families from \cite[Table 1]{moonen}.  Assume $p \nmid m$.   
Let $\nu \in \nu(B(\mu_m, \cf))$ be a Newton polygon occuring on $Z(m,N,a)$. Then $\nu$ occurs on $Z^0(m,N,a)$, 
meaning that $\nu$ occurs as the Newton polygon of the Jacobian of a smooth curve in the family unless either:
\begin{enumerate}
\item $\dim(Z(m,N,a))=1$, $\nu$ is supersingular, and $p$ is not sufficiently large; or
\item $\dim(Z(m,N,a))\geq 2$, and $\nu$ is supersingular.
\end{enumerate}
\end{theorem}

In cases (1) and (2), we prove in \cite[Theorem 6.1]{LMPT3} that there exists a smooth curve in the family which is supersingular for sufficiently large $p$.

\begin{proof}
For the family $M[1]$ (resp.\ $M[2]$), where $m=2$, 
the Shimura variety $Z(m,N,a)$ is simply ${\mathcal A}_1$ (resp.\ ${\mathcal A}_2$).
Thus all Newton polygons for $M[1]$ and $M[2]$ are already known to occur for Jacobians of smooth curves. 

Consider one of the families $M[3-20]$, where $m \geq 3$.		
For the $\mu$-ordinary Newton polygon $\nu_o$, the result follows from Proposition \ref{muord}.
Suppose that $\nu \in \nu(B(\mu_m, \cf))$ is not $\mu$-ordinary.
We distinguish two cases: $\dim(Z(m,N,a))=1$ and $\dim(Z(m,N,a))\geq 2$.

Suppose $\dim(Z(m,N,a))=1$. Then the set $\nu(B(\mu_m, \cf))$ contains exactly two polygons; 
since $\nu$ is not $\mu$-ordinary, then $\nu$ is the basic Newton polygon $\nu_b$. 
Direct computations (Lemma \ref{decomp}) show that if 
$\nu_b$ is not supersingular, then it is PEL-indecomposable and the result follows from Proposition \ref{indecomp}.
If $\nu_b$ is supersingular, by Proposition \ref{basic}, $\nu_b$ occurs for a smooth curve if $p$ is sufficiently large.

Suppose $\dim(Z(m,N,a))\geq 2$ and the basic Newton polygon $\nu_b\in \nu(B(\mu_m,\cf))$ is not supersingular. 
From the tables in Section \ref{sec_table}, this occurs exactly when $p\equiv 1\bmod m$.
Again, direct computations (Lemma~\ref{decomp}) show that $\nu_b$ is PEL-indecomposable. 
By Proposition \ref{indecomp}, $\nu_b$ occurs for the Jacobian of a smooth curve in the family, 
i.e., $Z^0[\nu_b]\neq \emptyset$, where $Z^0=Z^0(m,N,a)$.
Furthermore, note that $\nu(B(\mu_m, \cf))$ is totally ordered and 
$$\codim(Z^0[\nu_b],Z^0)=\# \nu(B(\mu_m,\cf))-1.$$
By the de Jong--Oort purity result for the stratification by Newton polygons \cite[Theorem~4.1]{JO}, 
we deduce that $Z^0[\nu]$ is non empty, for all $\nu\in \nu(B(\mu_m, \cf))$.
 
Suppose $\dim(Z(m,N,a))\geq 2$ and the basic Newton polygon $\nu_b\in \nu(B(\mu_m,\cf))$ is supersingular. The proof is complete unless $\nu(B(\mu_m,\cf))$ contains more than $2$ Newton polygons. From the tables in Section \ref{sec_table},
this happens only for the family $M[10]$ when $p \equiv 2 \bmod 3$, in which case $\nu(B(\mu_m, \cf))$ 
contains exactly three polygons: the $\mu$-ordinary $\nu_0$; 
the basic Newton polygon $\nu_b$; and $\eta=(1/4,3/4)$, which is PEL-indecomposable. 
By Proposition \ref{indecomp}, $\eta$ occurs for the Jacobian of a smooth curve in the family $M[10]$.
\end{proof}

\section{Tables}\label{sec_table} 
The following tables contain all the Newton polygons and the $\mu$-ordinary Dieudonn\'e modules
which occur for the cyclic covers of the projective line 
arising from the special families of Moonen.
The label $M[r]$ denotes the $r$th label in \cite[Table 1]{moonen}. 
We organize the tables for $M[1]-M[20]$ by the value of $m$ starting with the ones with $\phi(m) = 1,2$.

\begin{notation}
	The degree of the cover is $m$, the inertia type is $a=(a_1, \ldots, a_N)$, the 
	signature is $\cf=(f(1), \ldots, f(m-1))$.  The prime orbits are denoted by $\CO$.
\end{notation}

\begin{notation}
	Let $ord=\{0,1\}$ (resp.\ $ss=\{1/2,1/2\}$) be the ordinary (resp.\ supersingular) Newton polygon. 
	For $s,t\in \NN$, with $s<t$ and ${\rm gcd}(s,t)=1$, 
	let $(s/t, (t-s)/t)$ denote the Newton polygon with slopes $s/t$ and  $(t-s)/t$, each occuring with multiplicity $t$. 
\end{notation}

\begin{notation}
	As in Section \ref{SnotDM}, let $\EE$ denote the non-commutative ring generated by $F$ and $V$
	with $FV=VF=0$.  Let EO-type stand for Ekedahl-Oort type.  Define:
	
	\begin{itemize}
		\item $L=\EE/\EE(F, V-1) \oplus \EE/\EE(V, F-1)$; it has rank $2$, $p$-rank $1$, and EO-type
		$[1]$;
		\item $N_{r,1} = \EE/\EE(F^r-V^r)$; rank $2r$, $p$-rank $0$, $a$-number $1$,
		and EO-type $[0,1, \ldots, r-1]$;
		\item $N_{r,2} = \EE/\EE(F^{r-1} -V) \oplus \EE/\EE(V^{r-1} -F)$; it has rank $2r$, $p$-rank $0$, $a$-number $2$,
		and EO-type $[0,1, \ldots, r-3, r-2, r-2]$.
	\end{itemize}
\end{notation}

\subsection{Tables for $m$ with $m=2,3,4,6$}  \mbox{ } \\

Here are the tables for Moonen's special families when $\phi(m) \leq 2$.
Because of previous work, e.g., \cite{AP:gen, Pr:large, LMPT}, we do not produce any new examples of Newton polygons and 
Dieudonn\'e modules for smooth curves in this subsection.

When $\phi(m) \leq 2$, each orbit of the eigenspaces under Frobenius has length $1$ or $2$.  
By Examples \ref{Esiglen1} and \ref{Esiglen2}, this implies that the $\mu$-ordinary Dieudonn\'e module $N_o$ is determined by the 
$\mu$-ordinary Newton polygon $\nu_o$; if $\nu_o= ord^{a} \oplus ss^{b}$, 
then $N_o \simeq L^a \oplus N_{1,1}^b$.

For $m=2$ with $p$ odd:
Family $M[1]$ has $a=(1,1,1,1)$, $\cf=(1)$, and $ord$ and $ss$ occur.
Family $M[2]$ has $a=(1,1,1,1,1,1)$, $\cf=(2)$, and $ord^2$, $ord \oplus ss$, and $ss^2$ occur.

\begin{center}
	\begin{tabular}{  |c|c|c|c|c|  }
		\hline
		$m=3$ & 	$a$ & $\cf$ & $p \equiv 1 \bmod 3$ & $p \equiv 2 \bmod 3$ \\
		\hline
		$M[3]$ & $(1,1,2,2)$  &  $(1,1)$ &  $ord^2$ &  $ord^2$  \\
		&&& $ss^2$  & $ss^2$ \\
		\hline
		$M[6]$ & $(1,1,1,1,2) $ &  $(2,1)$ & $ ord^3$ & $ord^2 \oplus ss$ \\
		&&& $ord \oplus ss^2$  & $ss^3$ \\
		&&& $(1/3,2/3)$  &\\
		\hline
		$M[10]$ & $(1,1,1,1,1,1)$ & $(3,1)$ & $ ord^4$ & $ord^2 \oplus ss^2$\\
		&&& $ord^2 \oplus ss^2$  & $(1/4,3/4)$ \\
		&&& $ord \oplus (1/3,2/3)$  & $ss^4$ \\
		&&& $(1/4,3/4)$  &  \\
		\hline
	\end{tabular}
\end{center}

\begin{center}
	\begin{tabular}{  |c|c|c|c|c|  }
		\hline
		$m=4$ & 	$a$ & $\cf$ & $p \equiv 1 \bmod 4$ & $p \equiv 3 \bmod 4$ \\
		\hline
		$M[4]$ & $(1,2,2,3)$  &  $(1,0,1)$ &  $ord^2$ &  $ord^2$\\
		&&& $ss^2$  & $ss^2$ \\
		\hline
		$M[7]$ & $(1,1,1,1)$  & $(2,1,0)$ & $ord^3$ & $ord \oplus ss^2$ \\
		&&& $ord^2 \oplus ss$  & $ss^3$ \\
		\hline
		$M[8]$ & $(1,1,2,2,2) $  & $(2,0,1)$ & $ ord^3$ & $ord^2 \oplus ss$\\
		&&& $ord \oplus ss^2$  & $ss^3$ \\
		&&& $(1/3,2/3)$  &\\
		\hline
	\end{tabular}
\end{center}

\begin{center}
	\begin{tabular}{  |c|c|c|c|  }
		\hline
		$m=6$ & $p$ & $1 \bmod 6$ & $ 5 \bmod 6$\\
		\hline
		$a$ &\backslashbox{$\cf$}{$\CO$} & split & $(1,5),(2,4),(3)$ \\
		\hline
		(2,3,3,4)  &  (1,0,0,0,1) & $ord^2$ &  $ord^2$ \\
		$M[5]$ && $ss^2$ & $ss^2$\\
		\hline
		(1,3,4,4)  &  (1,1,0,0,1) & $ord^3$ &  $ord^2 \oplus ss $ \\
		$M[9]$&& $ord \oplus ss^2$ & $ss^3$\\
		\hline
		(1,1,1,3)  &  (2,1,1,0,0) & $ord^4$ &  $ord \oplus ss^3$ \\ 
		$M[12]$ & &  $ord^3 \oplus ss $ &    $ss^4$\\
		\hline
		(1,1,2,2)  &  (2,1,0,1,0) & $ord^4$ &  $ord^2 \oplus ss^2$\\
		$M[13]$ && $ ord^2 \oplus ss^2$ & $ss^4$\\
		\hline
		(2,2,2,3,3)  &  (2,0,0,1,1) & $ ord^4$ &  $ ord^2 \oplus ss^2$\\
		$M[14]$&& $ ord^2 \oplus ss^2$ & $ ss^4$\\
		&& $ord \oplus (1/3,2/3)$ &\\
		\hline
	\end{tabular}
\end{center}

\subsection{Tables for $m$ with $m=5,7,8,9,10,12$} \label{tablemlarge}\mbox{} \\

Here are the tables of Newton polygons for Moonen's special families when $\phi(m) > 2$.
We include the $\mu$-ordinary Dieudonn\'e module $N_o$ when the prime orbits have $\#\co > 2$; 
when $\# \co \leq 2$, then $N_o$ can be computed from Examples \ref{Esiglen1} and \ref{Edieumod2}.

\begin{center}
\begin{small}
	\resizebox{\columnwidth}{!}{
		\begin{tabular}{  |c|c|c|c|c|  }
			\hline
			$m=5$	& $p$ & $ 1 \bmod 5$ & $2,3 \bmod 5$ & $ 4 \bmod 5$\\
			\hline
			$a$ &\backslashbox{$\cf$}{$\CO$} & split & $(1,2,3,4)$ & $(1,4)$, $(2,3)$  \\
			\hline
			$(1,3,3,3)$  &  $(1,2,0,1)$ &  $ord^4$ &  $(1/4,3/4)$ & $ord^2 \oplus ss^2$\\ $M[11]$ & &$ord^2 \oplus ss^2$ & $ss^4$ &    $ss^4$\\
			\hline
			$(2,2,2,2,2) $ &  $(2,0,3,1)$ & $ ord^6$ &  $ (1/4,3/4) \oplus ss^2$ & $ ord^2 \oplus ss^4$\\
			$M[16]$&& $ ord^4 \oplus ss^2$ & $ ss^6$ & $ ss^6$\\
			&& $ord^3 \oplus (1/3,2/3)$ & &\\
			\hline
		\end{tabular}
	}
\end{small}
\end{center}

When $p \equiv 2,3 \bmod 5$: for $M[11]$, $N_o \simeq N_{4,2}$;
for $M[16]$, $N_o \simeq N_{4,2} \oplus N_{2,1}$.


\begin{small}
	
	\begin{center}
		\resizebox{\columnwidth}{!}{
			\begin{tabular}{  |c|c|c|c|c|c|  }
				\hline
				$m=7$ & $p$ & $1 \bmod 7$ & $2,4 \bmod 7$& $3,5 \bmod 7$ & $6 \bmod 7$\\
				\hline
				$a$ &\backslashbox{$\cf$}{$\CO$} & split & $(1,2,4),(3,5,6)$ & $(1,2,3,4,5,6)$ & $(1,6),(2,5),(3,4)$  \\
				\hline
				(2,4,4,4)  &  (1,2,0,2,0,1) & $ord^6$ &  $ord^3 \oplus (1/3,2/3)$ & $(1/3,2/3)^2$& $ord^2 \oplus ss^4 $ \\ 
				$M[17]$ & &  $ord^4 \oplus ss^2 $ &    $(1/6,5/6)$& $ss^6 $&$ss^6 $\\
				\hline
			\end{tabular}
		}
	\end{center}
\end{small}
\begin{small}
	
	By Example \ref{EOM17}, for $M[17]$:
	when $p \equiv 2,4 \bmod 7$, then $N_o \simeq L^3 \oplus N_{3,2}$;
	when $p \equiv 3,5 \bmod 7$,
	then $N_o \simeq \EE\langle e_5,e_6 \rangle/\EE(F^3 e_5 -V e_6, V e_5 - F e_6) 
	\oplus \EE\langle e'_3, e'_5 \rangle/\EE(Fe_3'- V^3 e_5', Ve_3'-Fe_5')$.
	
	\begin{center}
		\resizebox{\columnwidth}{!}{
			\begin{tabular}{  |c|c|c|c|c|c|  }
				\hline
				$m=8$ & $p$ & $1 \bmod 8$ & $ 3 \bmod 8$& $ 5\bmod 8$ & $ 7 \bmod 8$\\
				\hline
				\multirow{2}{*}{a}& \multirow{2}{*}{\backslashbox{$\cf$}{$\CO$}} & \multirow{2}{*}{split}  & $(1,3),(2,6)$ & $(1,5),(3,7)$ & $(1,7),(2,6)$  \\
				& &  & $(5,7),(4)$ & $(2),(4),(6)$ & $(3,5),(4)$  \\
				\hline
				(2,4,5,5)  &  (1,1,0,0,2,0,1) & $ord^5$ &  $ord^2 \oplus ss^3$ & $ord^3 \oplus ss^2$& $ord^2 \oplus ss^3$ \\  
				$M[15]$ & &  $ord^3 \oplus ss^2$ &   $(1/4,3/4) \oplus ss$& $ord \oplus (1/4,3/4) $&$ss^5 $\\
				\hline
			\end{tabular}
		}
	\end{center}
\end{small}

\begin{center}
	\begin{small}
		\resizebox{\columnwidth}{!}{
			\begin{tabular}{  |c|c|c|c|c|c|  }
				\hline
				$m=9$ & $p$ & $ 1 \bmod 9$ & $ 2,5 \bmod 9$& $4,7 \bmod 9$ & $8 \bmod 9$\\
				\hline
				\multirow{2}{*}{a}& \multirow{2}{*}{\backslashbox{$\cf$}{$\CO$}} & \multirow{2}{*}{split} & $(1,2,4,8,7,5)$ & $(1,4,7),(2,8,5)$ & $(1,8),(2,7)$  \\
				& &  & $(3,6)$ & $(3),(6)$ & $(4,5),(3,6)$  \\
				\hline
				(3,5,5,5)  &  (1,2,0,2,0,1,0,1) & $ord^7$ &  $(1/3,2/3)^2 \oplus ss$ & $ord \oplus (1/3,2/3)^2$& $ord^2 \oplus ss^5$ \\ 
				
				$M[19]$& &  $ord^5 \oplus ss^2$ &    $ss^7$& $ord \oplus ss^6 $&$ss^7 $\\
				\hline
			\end{tabular}
		}
	\end{small}
\end{center}

By Example \ref{EOM19}, for $M[19]$: when $p \equiv 4,7 \bmod 9$, then $N_o \simeq L \oplus N_{3,2}^2$;
when $p \equiv 2,5 \bmod 9$, then $N_o \simeq \EE/\EE(F^4-V^2) \oplus \EE/\EE(V^4-F^2) \oplus \EE/\EE(F-V)$.

\begin{center}
	
	\begin{tabular}{  |c|c|c|c|c|  }
		\hline
		$m=10$ & $p$ & $ 1 \bmod 10$ & $3,7 \bmod 10$ & $ 9 \bmod 10$\\
		\hline
		\multirow{2}{*}{a}& \multirow{2}{*}{\backslashbox{$\cf$}{$\CO$}} & \multirow{2}{*}{split} & $(1,3,9,7)$ & $(1,9),(2,8)$   \\
		& &  & $(2,6,8,4),(5)$ & $(3,7),(4,6),(5)$  \\
		\hline
		(3,5,6,6)  &  (1,1,0,1,0,0,2,0,1) & $ord^6$ &  $(1/4,3/4) \oplus ss^2$ &  $ord^2 \oplus ss^4$ \\  
		$M[18]$ & &  $ord^4 \oplus ss^2$ &  $ss^6 $&$ss^6 $\\
		\hline
	\end{tabular}
\end{center}

For $M[18]$:
when $p \equiv 3,7 \bmod 10$, then $N_o \simeq N_{4,2} \oplus N_{2,1}$.


\begin{center}
	
	\begin{small}
		\resizebox{\columnwidth}{!}{
			\begin{tabular}{  |c|c|c|c|c|c|  }
				\hline
				$m=12$ & $p$ & $1 \bmod 12$ & $5 \bmod 12$ & $ 7 \bmod 12$&  $ 11 \bmod 12$\\
				\hline
				\multirow{3}{*}{a}& \multirow{3}{*}{\backslashbox{$\cf$}{$\CO$}} & \multirow{3}{*}{split} &  &$(1,7),(3,9)$ & $(1,11),(4,8)$  \\
				& & & $(1,5),(2,10),(3)$ & $(2),(4),(5,11)$ & $(3,9),(2,10)$  \\
				&&  & $(4,8),(6),(7,11),(9)$ & $(6),(8),(10)$ & $(5,7),(6)$ \\
				\hline
				(4,6,7,7)  &  (1,1,0,1,0,0,2,0,1,0,1) & $ord^7$ &  $ord^3 \oplus ss^4$ & $ord^4 \oplus ss^3$& $ord^2 \oplus ss^5$ \\  
				$M[20]$ & &  $ord^5 \oplus ss^2$ &    $ord \oplus (1/4,3/4) \oplus ss^2$& $ord^2 \oplus ss^5 $&$ss^7 $\\
				\hline
				
			\end{tabular}
		}
	\end{small}
\end{center}

\subsection{PEL-decomposable Newton polygons}

The next lemma contains the data of all the PEL-decomposable Newton polygons for families $M[3-20]$ from \cite[Table~1]{moonen}. We refer to Remark \ref{Raddmiss} and Notation \ref{Nadddeg}
for explanation and to Example \ref{Em15add} for a sample calculation.

\begin{lemma} \label{degen}\label{decomp}
	In the following tables, for the special families from 
	\cite[Table 1]{moonen},
	the fourth column lists the degenerations of compact type of the inertia type and  
	the last column lists the PEL-decomposable Newton polygons 
	for the family under the given congruence condition on $p$ modulo $m$
	(excluding the $\mu$-ordinary Newton polygon when $p\equiv 1\bmod m$
	\footnote{The $\mu$-ordinary Newton polygon for $p\equiv 1\bmod m$ is ordinary, and always PEL-decomposable.}).
	
\begin{small}
\begin{center}	
	\begin{tabular}{ |c|c|c|c|c|  }
		\hline
label& m & inertia type & degeneration & PEL-decomp. NP (congruence class) \\
\hline
$M[3]$ & 3 & (1,1,2,2) & (1,1,1)+(2,2,2) & $ss^2$, $p \equiv 2$ \\
\hline
\multirow{3}{*}{$M[6]$} & \multirow{3}{*}{3}& \multirow{3}{*}{(1,1,1,1,2)} &\multirow{3}{*} {(1,1,1)+(1,1,2,2)} &  $ord \oplus ss^2$, $p \equiv 1$ \\
&&&& $ord^2 \oplus ss$, $p \equiv 2$\\ &&&& $ss^3$, $p \equiv 2$ \\
\hline
\multirow{3}{*}{$M[10]$} &\multirow{3}{*}{3}& \multirow{2}{*}{(1,1,1,1,1,1)} & \multirow{3}{*}{(1,1,1) + (2,1,1,1,1)} & $ord^2 \oplus ss^2$, $p \equiv 1, 2$ \\
&&&&   $ord\oplus (1/3,2/3)$, $p \equiv 1$\\ &&&& $ss^4$, $p \equiv 2$ \\
\hline
%
$M[4]$ & 4 & (1,2,2,3) & (1,1,2)+(2,3,3) & $ss^2$, $p \equiv 3$ \\
\hline
$M[7]$ &4& (1,1,1,1) & None & None \\
\hline
\multirow{3}{*}{$M[8]$} & \multirow{3}{*}{4}& \multirow{3}{*}{(1,1,2,2,2)} & {(1,2,2,3)+(1,1,2)} &$ord \oplus ss^2$, $p \equiv 1$ \\
&&&{(1,1,2) + ${\rm Ind}_2^4$(1,1,1,1)} & $ord^2 \oplus ss$, $p \equiv 3$\\
&&&&$\sss^3$, $p \equiv 3$ \\
%
\hline
$M[5]$ & 6 & (2,3,3,4) & (1,2,3)+(3,4,5) & $ss^2$, $p \equiv 5$ \\
\hline

\multirow{2}{*}{$M[9]$} & \multirow{2}{*}{6} & \multirow{2}{*}{(1,3,4,4)} & (1,1,4)+(3,4,5) &  \multirow{2}{*}{$ss^3$, $p \equiv 5$} \\
&&& (1,2,3) + ${\rm Ind}_3^6$(2,2,2) & \\
\hline

$M[12]$  & 6 & (1,1,1,3) & (1,1,4) + (2,1,3) &$ss^4$, $p \equiv 5$ \\
\hline 
$M[13]$  &6& (1,1,2,2) & 
(1,1,4) + ${\rm Ind}_3^6$(1,1,1) &$ss^4$, $p \equiv 5$ \\ 
\hline
\multirow{2}{*}{$M[14]$}  &\multirow{2}{*}{6}& \multirow{2}{*}{(2,2,2,3,3)} & (2,3,1) + (5,3,2,2) & $ord^2 \oplus ss^2$, $p \equiv 1,5$ \\
&&& $ (2,3,3,4) + {\rm Ind}_3^6(1,1,1)$ & $ss^4$, $p \equiv 5$ \\
\hline 
\end{tabular}
\end{center}
\end{small}

\begin{small}
	\begin{center}	
		\begin{tabular}{ |c|c|c|c|c|  }
			\hline
			label & $m$ & $a$ & degenerations & PEL-dec. NP, congruence on $p$ \\
			\hline	
$M[11]$ & 5 & (1,3,3,3) & (1,3,1) + (4,3,3) &$ss^4$, $p \not\equiv 1$ \\		
\hline
						\multirow{4}{*}{$M[16]$} & \multirow{4}{*}{5} & \multirow{4}{*}{(2,2,2,2,2)} & \multirow{4}{*}{(2,2,2,4) + (1,2,2)} 
&$(1/4,3/4)\oplus ss^2$, $p \equiv 2,3$\\
			&&&&$ord^2 \oplus ss^4$, $p \equiv 4$\\
			&  & & & $ord^4 \oplus ss^2$, $p \equiv 1$\\
			&&&& $ss^6$, $p \not\equiv 1$ \\
			\hline
			\multirow{2}{*}{$M[17]$} & \multirow{2}{*}{7} &\multirow{2}{*}{(2,4,4,4)}  & \multirow{2}{*}{(4,4,6) + (1,4,2)}& $ord^3 \oplus (1/3,2/3)$, $p \equiv 2,4$\\
			&&&& $ss^6$, $p \equiv 3,5,6$ \\
			\hline
			\multirow{3}{*}{$M[15]$} & \multirow{3}{*}{8} & \multirow{3}{*}{(2,4,5,5)} & 
			(5,5,6) + ${\rm Ind}_4^8$(1,2,1) & $ord^2 \oplus ss^3$, $p \equiv 3$\\ 
			&&& (5,2,1)+(7,4,5) & $ord^3 \oplus ss^2$, $p \equiv 5$\\
			&&&& $ss^5$, $p \equiv 7$  \\
			\hline 

			\multirow{2}{*}{$M[19]$} & \multirow{2}{*}{9} & \multirow{2}{*}{(3,5,5,5)} & \multirow{2}{*}{(3,5,1) + (8,5,5)} &$ ord\oplus (1/3,2/3)^2 $, $p \equiv 4,7$\\
			&&&& $ss^7$, $p \equiv 2,5,8$ \\
			\hline
			\multirow{2}{*}{$M[18]$} & \multirow{2}{*}{10} & \multirow{2}{*}{(3,5,6,6)} & 
			(3,5,2) + ${\rm Ind}_5^{10}$ (4,3,3) & \multirow{2}{*}{$ss^6$, $p \not\equiv 1$} \\
			& & & (6,3,1) + (9,5,6) &\\ 
			\hline
			
			\multirow{3}{*}{$M[20]$} & \multirow{3}{*}{12} & \multirow{3}{*}{(4,6,7,7)} & 
			 (7,7,10) + ${\rm Ind}_6^{12}$ (1,2,3) & $ord^3 \oplus ss^4$, $p \equiv 5$\\
			&&&(7,4,1) + (11,6,7)& $ord^4 \oplus ss^3$, $p \equiv 7$ \\
			&&& &$ss^7$, $p \equiv 11$  \\
			\hline
		\end{tabular}
	\end{center}
\end{small}

	\end{lemma}


\section{Applications} \label{Sapplication}

The method of the previous sections produces numerous Newton polygons and mod $p$ Dieudonn\'e modules
that occur for the Jacobian of a smooth curve.
In Section \ref{Snewapp}, we collect a list of these, focusing on the ones which are new, 
and prove a result about curves of arbitrary large genus whose Newton polygon has slopes $1/6, 5/6$.
In Section \ref{Snonsp}, for an infinite sequence of $g \in \NN$ and a set of primes of density $1/2$, 
we produce an explicit family of smooth curves of genus $g$ such that the 
multiplicity of the slope $1/2$ in the Newton polygon is at least $4\sqrt{g}$.

\subsection{Newton polygons and Dieudonn\'e modules arising from special families} \label{Snewapp}

The label $M[r]$ denotes the $r$th label in \cite[Table 1]{moonen}.

\begin{theorem}[Theorem \ref{Tintro1}]\label{Tapp1} 
There exists a smooth supersingular curve of genus $g$ defined over $\overline{\FF}_p$ 
for $p$ sufficiently large in the congruence class in the following cases:
\begin{itemize}
\item when $g=5$, $p \equiv 7 \bmod 8$, from $M[15]$;
\item when $g=6$, $p \equiv 2,3,4 \bmod 5$, from $M[18]$;
\item when $g=7$, $p \equiv 2 \bmod 3$, from $M[19,20]$.
\end{itemize}
\end{theorem}

\begin{proof}
The results appear in Section \ref{sec_table} in the basic loci of families. 
By Proposition \ref{basic}, for $p$ sufficiently large in each congruence class, each of the supersingular Newton polygons above appears as the Newton polygon of the Jacobian of a smooth curve.
\end{proof}

\begin{remark}
\begin{enumerate}
\item In Theorem \ref{Tapp1}, we could also include: genus 4 when $p \equiv 2 \bmod 3$ from 
$M[12-13]$ or $p \equiv 2,3,4 \bmod 5$ from $M[11]$;
and genus $6$ when $p \equiv 3,5,6 \bmod 7$ from $M[17]$. 
We do not include these cases since they already appear in \cite[Theorem 1.1]{LMPT} for all $p$ satisfying the congruence condition.

\item We would like to thank Voight and Long for pointing out that the genus $6$ case of Theorem \ref{Tapp1}
can likely also be proven (without the hypothesis that $p$ is sufficiently large) using truncated hypergeometric functions.
\end{enumerate}
\end{remark}

We collect new examples of Newton polygons and Dieudonn\'e modules
from the $\mu$-ordinary locus.
The two Newton polygons with $*$ appear earlier in \cite{AP:gen, Pr:large}.

\begin{theorem} \label{TmuordNP}
There exists a smooth curve of genus $g$ defined over $\overline{\FF}_p$ with the given Newton polygon $NP$
and mod $p$ Dieudonn\'e module $DM$ in the following cases:
\begin{center}
	\begin{tabular}{ |c|c|c|c|c|  }
		\hline
		genus &	NP & DM & congruence on $p$ & where \\
		\hline
		4 & $*(1/4, 3/4)$ & $N_{4,2}$ & $2,3 \bmod 5$ & $M[11]$\\
		\hline
		5 & $ord^2 \oplus ss^3$ & $L^2 \oplus N_{1,1}^3$ & $3,7 \bmod 8$ & $M[15]$\\
		\hline
		6  & $(1/4,3/4) \oplus ss^2$  & $N_{4,2} \oplus N_{2,1}$ & $2,3 \bmod 5$  & $M[16, 18]$  \\
		\hline
		\multirow{2}{*}{6}  & \multirow{2}{*}{$ord^2 \oplus ss^4$}  & \multirow{2}{*}{$L^2 \oplus N_{1,1}^4$} & $4 \bmod 5$, $6 \bmod 7$,  & $M[16,17]$ \\
		&&& or $9 \bmod 10$ & $M[18]$ \\
		\hline
		6 & $*ord^3 \oplus (1/3,2/3)$ & $L^3 \oplus N_{3,2}$ & $2,4 \bmod 7$ & $M[17]$ \\
		\hline
		6  & $(1/3,2/3)^2 $  & Example \ref{EOM17}(2) & $3,5 \bmod 7$  & $M[17]$  \\
		\hline
		7  & $(1/3,2/3)^2 \oplus ss$  & Example \ref{EOM19}(1) & $2,5 \bmod 9$  & $M[19]$  \\
		\hline
	
		7  & $ord \oplus (1/3,2/3)^2$  & $L \oplus N_{3,2}^2$ & $4,7 \bmod 9$  & $M[19]$ \\
		\hline
		7 & $ord^4 \oplus ss^3$ & $L^4 \oplus N_{1,1}^3$ & $7 \bmod 12$ & $M[20]$\\
		\hline
		7  & $ord^3 \oplus ss^4$  & $L^3 \oplus N_{1,1}^4$ & $5 \bmod 12$  & $M[20]$ \\
		\hline
		7  & $ord^2 \oplus ss^5$  & $L^2 \oplus N_{1,1}^5$ & $8 \bmod 9$ or $11 \bmod 12$  & $M[19,20]$ \\
		\hline

			\end{tabular}
\end{center}
\end{theorem}

\begin{proof}
The result follows from Proposition \ref{muord} since
these Newton polygons and Dieudonn\'e modules occur in $\mu$-ordinary cases in Section \ref{sec_table}
under these congruence conditions. 
\end{proof}

\begin{theorem} \label{TbasicNP}
There exists a smooth curve of genus $g$ defined over $\overline{\FF}_p$ with the given Newton polygon $NP$
in the following cases:
\begin{center}
	\begin{tabular}{ |c|c|c|c|  }
		\hline
		genus &	NP & congruence on $p$ & where \\
		\hline
		5  & $(1/4,3/4) \oplus ss $  & $3 \bmod 8$  & $M[15]$  \\
		\hline
		6  & $(1/6,5/6)$  & $2,4 \bmod 7$  & $M[17]$  \\
		\hline
		7  & $ord \oplus ss^6$  & $4,7 \bmod 9$  & $M[19]$ \\
		\hline
		7  & $ord^2 \oplus ss^5$  & $7 \bmod 12$  & $M[20]$ \\
		\hline		
		7  & $ord \oplus (1/4,3/4) \oplus ss^2$  & $5 \bmod 12$  & $M[20]$ \\
		\hline
			\end{tabular}
\end{center}

\end{theorem}

\begin{proof}
These Newton polygons appear in basic loci cases from Section \ref{sec_table}
under the given congruence conditions. 
By Lemma \ref{degen}, these cases are PEL-indecomposable. 
Thus, by Proposition \ref{dec}, they occur for Jacobians of smooth curves.
\end{proof}

The $p$-divisible group $G_{1, d-1} \oplus G_{d-1,1}$ has slopes $1/d, (d-1)/d$.
The next result was proven for all $p$, when $d=2$ \cite[Theorem~2.6]{FVdG}; 
$d=3$ \cite[Theorem~4.3]{Pr:large}; $d=4$ \cite[Corollary~5.6]{AP:gen}; 
and (under congruence conditions on $p$) when $d=5$ or $d=11$ in \cite[Theorem~5.6]{LMPT}.

\begin{theorem} \label{Tapp2} 
When $d=6$ and $p \equiv 2,4 \bmod 7$, for all $g \geq d$, 
there exists a smooth curve of genus $g$ defined over $\overline{\FF}_p$
whose Jacobian has $p$-divisible group isogenous
to $(G_{1, d-1} \oplus G_{d-1,1})\oplus (G_{0,1} \oplus G_{1,0})^{g-d}$:
\end{theorem}

\begin{proof}
From the basic locus for $M[17]$ when $m=7$ in Section \ref{sec_table},
there exists a curve of genus $d=6$ 
defined over $\overline{\FF}_p$ whose $p$-divisible group is isogenous to $G_{1, d-1} \oplus G_{d-1,1}$.
By Proposition \ref{indecomp}, this curve is smooth. 
The Newton polygon for $G_{1, d-1} \oplus G_{d-1,1}$ is the lowest Newton polygon in dimension $d$ with $p$-rank $0$.
Thus there is at least one component of the $p$-rank $0$ stratum of $\CM_d$ whose 
generic geometric point has $p$-divisible group isogenous to $G_{1, d-1} \oplus G_{d-1,1}$.
The result is then immediate from \cite[Corollary~6.4]{priesCurrent}.
\end{proof}

\subsection{A non-special family} \label{Snonsp} \mbox{}\\

Let QR (resp.\ QNR) be an abbreviation for quadratic residue (resp.\ non-residue).

\begin{notation} \label{Nnonspecial}
Let $m > 7$ be a prime such that $m \equiv 3 \bmod 4$.
Let $p \geq m(m-7)/2$ be a prime which is a QNR modulo $m$.
Let $N=(m-1)/2$.
Let $\alpha_m=(a_1, \ldots, a_N)$ be an ordering of the QRs modulo $m$.
\end{notation}

The triple $(m, N, \alpha_m)$ is a monodromy datum because $\sum_{i=1}^N a_i \equiv 0 \bmod m$. 

Consider the family $C=C(m,N,\alpha_m) \to U$ of curves defined as in \eqref{EformulaC}.
For $t \in U$, the genus of the curve $C_t$ is $g_m=(m-5)(m-1)/4$ by \eqref{Egenus}.
Recall that $Z(m,N,\alpha_m)$ is the closure in $\CA_{g_m}$ of the image of $C$ under the Torelli map, and
$S(m,N,\alpha_m)$ is the smallest Shimura subvariety in $\CA_{g_m}$ containing $Z(m,N,\alpha_m)$.

\begin{proposition}\label{P121}
Consider the monodromy datum $(m,N, \alpha_m)$ with $m, p, N, \alpha_m$ as in Notation \ref{Nnonspecial}.
There are explicitly computable constants $E(1), E(2) \in \NN$, with $E(1) > E(2)$ and $E(1)+E(2)=(m-5)/2$,
such that the $\mu$-ordinary Newton polygon $\nu_o$ and $\mu$-ordinary mod $p$ Dieudonn\'e module $N_o$
of the reduction modulo $p$ of $S(m,N,\alpha_m)$ 
are
\begin{eqnarray} \label{muordnonspecial}
\nu_o & = & (ord^{2E(2)} \oplus ss^{E(1)-E(2)})^{(m-1)/2},\\
N_o & = & (L^{2E(2)} \oplus N_{1,1}^{E(1)-E(2)})^{(m-1)/2}.
\end{eqnarray}
\end{proposition}

\begin{proof}
Let $n\in (\ZZ/m\ZZ)^*$.  Since the inertia type $\alpha_m$ consists of all the QRs,
by \eqref{DMeqn}, there exist $c_1, c_2\in {\mathbb Z}^{>0}$ such that
$\cf(\tau_n)=c_1$ when $n$ is a QR
and $\cf(\tau_n)=c_2$ when $n$ is a QNR.
Since $m \equiv 3 \bmod 4$, exactly one of $n$ and $-n$ is a QR.
So $g(\tau_n)=\cf(\tau_n)+\cf(\tau_{-n})=c_1+c_2=(m-5)/2$. 
Also $c_1\neq c_2$ since $(m-5)/2$ is odd.

Since $p$ is a QNR
in $(\ZZ / m\ZZ)^*$, each orbit $\co$ has the same number of 
QRs and QNRs.
By \S\ref{muordformula}, $s(\co)=2$, $E(0)=c_1+c_2$, $E(1)=\max \{ c_1, c_2\}$ and $E(2)=\min \{ c_1, c_2\}$.
By \eqref{slope}-\eqref{multiplicity}:
\begin{align*}
&\lambda(0)=0, \  m(\lambda(0)=(\#\co) \cdot E(2);\\
&\lambda(1)=1/2, \ m(\lambda(1))  = (\#\co)\cdot(E(1)-E(2)); \ {\rm and}\\
&\lambda(2)=1, \ m(\lambda(2))  = (\#\co)\cdot E(2).
\end{align*}
Hence, $\mu(\co)=(ord^{2E(2)} \oplus ss^{E(1)-E(2)})^{\#\co/2}$ and $\nu_o=\sum \mu(\co)=(ord^{2E(2)} \oplus ss^{E(1)-E(2)})^{(m-1)/2}$. 
Because $E(1)-E(2) \ne 0$, the multiplicity of the slope $1/2$ in $\nu_o$ is at least $m-1$.  

Since $s(\co)=2$ for any orbit $\co$ and $E(1)=\max \{ c_1, c_2\}$ and $E(2)=\min \{ c_1, c_2\}$ are independent of $\co$, 
then $N_o = (L^{2E(2)} \oplus N_{1,1}^{E(1)-E(2)})^{(m-1)/2}$ by Example \ref{Edieumod2}.
\end{proof}

For $m > 7$, the image of the Torelli morphism is not open and dense in $S(m,N,\alpha_m)$.
This makes it extremely difficult to determine the generic Newton polygon for the curve $C \to U$; 
In this case, a result of Bouw allows us to prove that the $\mu$-ordinary Newton polygon $\nu_o$ 
occurs for the curves in the family.

\begin{theorem}\label{last}
Consider the monodromy datum $(m, N, \alpha_m)$ with $m, p,N, \alpha_m$ as in Notation \ref{Nnonspecial}.
Consider the family $C(m,N,\alpha_m) \to U$ of curves of genus $g_m$ defined as in \eqref{EformulaC}. 
For $t$ in an open dense subset of $U$, the curve $C_t$ is smooth and has Newton polygon $\mu_o$ and mod $p$ Dieudonn\'e module $N_o$
as in \eqref{muordnonspecial}.
In particular, the Newton polygon only has slopes $0$, $1/2$, and $1$ and the multiplicity of the slope $1/2$ is at least 
$m-1 \geq 2\sqrt{g_m}$.
\end{theorem}

\begin{proof}
Let $S_p$ denote the reduction modulo $p$ of $S(m, N, \alpha_m)$; it has 
$\mu$-ordinary Newton polygon $\mu_o$ by Proposition \ref{P121}. 
The Newton polygon $\mu_o$ is the lowest among all Newton polygons of the same $p$-rank.
Thus in $S_p$,
the maximal $p$-rank stratum agrees with the $\mu$-ordinary Newton polygon (and Dieudonn\'e module) stratum.  
In particular, the Jacobian of $C_t$ has Newton polygon $\nu_o$ if and only if it has maximal possible $p$-rank.
By \cite[Theorem 6.1]{Bouwprank}, for $p \geq m(N-3)=m(m-7)/2$, 
the maximal possible $p$-rank is achieved on an open dense subset of $U$.
\end{proof}

In the first few examples, the multiplicity of the slope $1/2$ in the Newton polygon
matches the lower bound $(m-1)/2$.

\begin{example}
	When $m=7$ and $a=(1,2,4)$, then $c_1=1$ and $c_2=0$. 
Thus, for $p \equiv 3,5,6 \bmod 7$, the Newton polygon is $ss^3$.
\end{example}

\begin{example}
	When $m=11$ and $a=(1,3,4,5,9)$, then $c_1=2$ and $c_2=1$. 
Thus, for $p \equiv 2,6,7,8,10 \bmod 11$, the Newton polygon is $ord^{10} \oplus ss^5$.
\end{example}

\begin{example}
	When $m=19$ and $a=(1,4,5,6,7,9,11,16,17)$, then $c_1=4$ and $c_2=3$. 
Thus, for $p \equiv 2,3,8,10,12,13,14,15,18 \bmod 19$, the Newton polygon is $ord^{54} \oplus ss^9$.
\end{example}

\section{Appendix: Bounding the number of irreducible components of the basic locus of simple Shimura varieties}
In this section, we study the basic locus of a certain type of unitary Shimura variety.  Under some natural
restrictions on the prime $p$, 
we prove that the number of irreducible components of the basic locus of its reduction modulo $p$ is unbounded as $p$ goes to $\infty$.

\subsection{Notation}

Let $E$ be a CM field, Galois over $\QQ$, and  $F$  its maximal totally real subfield. 
Recall from Section \ref{sec_Sh_GSp}, that $V=\QQ^{2g}$ has a standard symplectic form $\Psi$.
There is an action of $E$ on $V$ which is compatible with $\Psi$. In other words, if we view $V$ as an $E$-vector space, 
the symplectic form on $V$ naturally induces an $E/F$-Hermitian form $\psi$ on $V$.

Using the notation in Section \ref{sec_Sh_PEL}, we consider the PEL-Shimura datum $(H_B, h)$ for $B=E$.
In the following, we write $G=H_E$.   
Note that we have an exact sequence
\[1\rightarrow \Res^F_\QQ U(V,\psi)\rightarrow G\rightarrow \GG_m \rightarrow 1\]
where $U(V,\psi)$ is the unitary group over $F$ with respect to $\psi$. 
For example,  in the notation of Section \ref{sec_Sh_PEL}, when $E=K_m$ is a cyclotomic field, 
{for each embedding} $\tau:K_m\rightarrow \CC$, then 
the signature of $\psi$ at the real place of $F$ induced by $\tau$ is $(\frf(\tau), \frf(\tau^*))$.
Note that the reflex field of the Shimura datum $(G,h)$ is contained in $E$. 

For any rational prime $p$, we fix a prime $\p$ of $E$ above $p$. In the following, with some abuse of notation, we still denote by $\p$ the corresponding primes of $F$ and of the reflex field.  
We write ${\mathbb A}_f$ for the finite adeles of $\mathbb Q$.
Let $K\subset G(\bA_f)$ be an open compact subgroup, and denote by $\Sh (G,h)$  the associated Shimura variety of  level $K$.
Assume $p$ is unramified in $E$, and $K=K_pK^p\subset G(\QQ_p)\times G(\bA^p_f)$ with $K_p$ hyperspecial.\footnote{If $p$ is unramified in $E$, then $G_{\QQ_p}$ admits a smooth reductive model $\mathcal G$ over $\ZZ_p$, and $K_p={\mathcal G}(\ZZ_p)$ is hyperspecial.} Note that this assumption holds for any prime $p$ sufficiently large.
Then, we denote 
by $\cS$ the canonical integral model of $\Sh(G,h)$  at $\p$, and write $\cS_\p$ for the $\bmod$  $\p$ reduction of $\cS$,
and  $\cS_\p(\nu_b)$ for its basic locus.

\subsection{Main theorem}

\begin{theorem}\label{prop_irredcomp}
	Assume that the signature of the unitary group $U(V,\psi)$ is $(1,n-1)$ or $(n-1,1)$ at one real place of $F$ and is $(0,n)$ or $(n,0)$ at any other real place. 
	If $\p\mid p$ is inert in $E/F$, we further assume that $n$ is even.
	Then for any such prime $\mathfrak p$, the number of (geometrically) irreducible components of
	${\mathcal S}_\p(\nu_b)$ grows to infinity with $p$.
\end{theorem}

\begin{remark}
	When $n$ is odd or when the signature has another form, the statement of Theorem \ref{prop_irredcomp} does not hold in general. For example, when the center $Z_G$ of $G$ is connected, Xiao and Zhu show that if the dimension of the basic locus is half the dimension of a Hodge type Shimura variety, then the number of its irreducible components is the same for all unramified primes \cite[Lemma 1.1.3, Theorem 1.1.4 (1), Remark 1.1.5 (2), Proposition 7.4.2]{XZ}. 
	This dimension requirement is satisfied for any unitary Shimura variety with the signature as in Theorem \ref{prop_irredcomp} and $n$ odd at inert prime $\p$. For more examples, see  \cite[Remark~4.2.11]{XZ}.
	
	Related material is in \cite{yumass}. 
\end{remark}

\begin{corollary}\label{cor_irredcomp}
	The statement of Theorem \ref{prop_irredcomp}
	holds true for any connected component of a product of finitely many unitary Shimura varieties satisfying the assumptions of Theorem \ref{prop_irredcomp}. 
\end{corollary}

\begin{definition}\textbf{The group of self-isogenies and certain open compact subgroups.}\label{ISO}
	Fix a point $x$ in ${\mathcal S}_\p(\nu_b)$, and let $A_x$ denote the associated abelian variety, endowed with the
	additional PEL-structure. We write $I=I_x$ for the group of quasi-isogenies of $A_x$ which 
	are compatible with the $E$-action and the $\QQ$-polarization.
	Then $I$ is an inner form of the algebraic reductive group $G$. Furthermore, $I({\mathbb Q}_\ell)=G({\mathbb Q}_\ell)$ at any finite prime $\ell\not= p$, and $I(\RR)/\RR^*$ is compact (due to the positivity of the Rosati involution). 
	
	In the following, we construct an open compact subgroup $C_p$ of $I(\QQ_p)$ such that the mass of $C_pK^p$ gives a lower bound of the number of irreducible components of $\cS_\p(\nu_b)$.
	
	\begin{enumerate}
		\item
		Assume $\p$ is split in $E/F$. Then, $I(\QQ_p)$ decomposes as $I(\QQ_p)=\QQ_p^\times\prod_{v|p}I_v$,\footnote{As a convention in the appendix, if we write $G=H_1H_2$, it means that $H_1,H_2$ are subgroups of $G$ and every element $g$ in $G$ can be written as $h_1h_2$ for $h_i\in H_i$, $i=1,2$ with \emph{no} assumption that such decomposition is unique. For instance, we do not assume that $H_1\cap H_2=\{1\}$.} where $v$ runs through all places of $F$ above $p$. Furthermore, $I_v\cong\GL_{n}(F_v)$ for $v\neq \p$, and  $I_\p$ is isomorphic to $D^\times_{F_\p,n}$ the division algebra over $F_\p$ with invariant $1/n$. We define $C_p$ as the maximal compact subgroup of $I(\QQ_p)$ given by 
		\[C_p=\ZZ_p^\times \cO^\times_{D_{F_\p,n}}\prod_{v|p, v\neq \p}\GL_n(\cO_{F_v}),\] where $\cO^\times_{D_{F_\p,n}}\subset D^\times_{F_\p,n}$ is the subgroup of elements with norm in $\cO_{F_\p}^\times$. 
		(For more details, see for example \cite[Chapter II.1]{HT}). 
		\item
		Assume $\p$ is inert in $E/F$. 
		We generalize the discussion in \cite[\S 2]{vollaard}, where Vollaard treats the case $F=\QQ$. Here, we follow \cite{XZ2} and recall that $n$ is even in this case.

		Let $V^\bullet$ be the $n$-dimensional $E/F$-Hermitian space such that
		\begin{itemize}
			\item $V^\bullet\otimes_F \bA_{F,f}^\p\cong V\otimes_F \bA_{F,f}^\p$ as Hermitian spaces; (we fix such an isomorphism);
			\item $V^\bullet$ has signature $(n,0)$ or $(0,n)$ at all archimedean places (more precisely, we only change the signature by $1$ at the one indefinite place of $V$),
			\item $V^\bullet \otimes_F F_\p$ is a ramified Hermitian space over $E_\p/F_\p$. 
		\end{itemize}
		
		Fix an $\cO_E\otimes \ZZ_p$-lattice $\Lambda^\bullet_p\subset V^\bullet \otimes_{\QQ}\QQ_p$, such that the dual $\Lambda^{\bullet, \vee}_p$ of $\Lambda^\bullet_p$ with respect to the Hermitian form satisfies $\Lambda^\bullet_p\subset \Lambda^{\bullet,\vee}_p$ and $\Lambda^{\bullet,\vee}_p/\Lambda^\bullet_p\simeq \cO_E/\p$ (such lattice exists due to the above assumption on $V^\bullet$). Then $I$ is the unitary similitude group of $V^\bullet$\footnote{Step (2) of the inert case of the proof of Proposition \ref{prop_double_coset} contains a proof for this well-known fact.} and we define $C_p\subset I(\QQ_p)$ to be the stabilizer group of $\Lambda^\bullet_p$.
		
		
	\end{enumerate}
\end{definition}

\begin{proposition}\label{prop_double_coset}
	The number of irreducible components of ${\mathcal S}(\nu_b)_{\mathfrak p}$ is bounded below by the mass of $C_pK^p$,
	$$m(C_pK^p):=\# I({\mathbb Q})\backslash I({\mathbb A}_f) /C_pK^p.$$
\end{proposition}

\begin{proof}
	By Rapoport and Zink's $p$-adic uniformization theorem \cite[Theorem 6.30]{RZ}, to prove the statement it suffices to show that the irreducible components (of a subset) of the Rapoport--Zink space $RZ$ {are} indexed by $I(\QQ_p)/C_p$.
	
	Assume $\p$ is split in $E/F$. 
	By Theorem \ref{thm_VW}, there are $n$ possible Newton polygons on $\cS_\p$ and hence by \cite[Theorem 1.1]{hamacher}, the basic locus is $0$-dimensional. Note that the group $C_p$ is the stabilizer of the Dieudonn\'e lattice of $A_x$. Hence $I(\QQ_p)$ acts on $RZ$ with stabilizer $C_p$ and the number of irreducible components is bounded below by $I(\QQ_p)/C_p$.
	
	Assume $\p$ is inert in $E/F$.
	For $F=\QQ$, {the statement} is \cite[Theorem 5.2 (2), Prop. 6.3]{vollaardwedhorn}. 
	For $F$ totally split at $p$, the statement is proved in \cite{XZ2}. Here, we sketch a proof that generalizes \cite{vollaardwedhorn} and \cite{XZ2}. The proof is in three steps. 
	\begin{enumerate}
		\item Construct a $0$-dimensional Shimura variety $\Sh (GU(V^\bullet))$ of level $C_pK^p$, parametrizing abelian varieties $A^\bullet$ of dimension $[F:\QQ]n$  with $\cO_E$-action and polarization $\lambda^\bullet$ satisfying Kottwitz's determinant condition. Here we assume that the polarization $\lambda^\bullet$ admits the same polarization type as in $\Sh (G,h)$ outside $\p$; at $\p$ we assume that $(\ker \lambda^\bullet)[p^\infty]=(\ker \lambda^\bullet)[\p]$ and that the latter is a finite flat group scheme of order $\#\cO_E/\p$.
		To describe the Kottwitz's condition, let $\Phi\subset \hom (E,\CC)$ be the subset of $\tau$ such that the signature $(\cf(\tau),\cf(\tau^*))$ of $V^\bullet$ is $(n,0)$. Note that $\Phi\sqcup \Phi^*=\hom(E,\CC)$. Kottwitz's determinant condition says that the characteristic polynomial of the action $b\in \cO_E$ on $\Lie(A^\bullet)$ is given by $\prod_{\tau\in \Phi}(x-\tau(b))^n$. 
		\item We will
		construct a Deligne--Lusztig variety $DL$ and, for a fixed $A^\bullet$ as in (1),  a family of abelian varieties $A$ in $\cS_\p(\nu_b)$ together with a universal isogeny from $A^\bullet$ parametrized by $DL^{\perf}$, the perfection of $DL$. 
		
		Let $\tau_0$  denote the unique indefinite real place for $U(V)$ and let $f$ denote the inertia degree of $\p$ in $F$. Consider the $\sigma$-orbit $\co_{\tau_0}=\{ \tau_0, \sigma\tau_0, \cdots, \sigma^f\tau_0=\tau_0^*,\sigma\tau_0^*, \cdots, \sigma^{f-1}\tau_0^*\}$. For $\tau\in\co_{\tau_0}$,
			define $\widetilde{\Frob}^f=F_f\circ\cdots F_1$, where $F_i:H_1^{\rm dR}(A^\bullet)_{\sigma^{i-1}\tau_0}\rightarrow H_1^{\rm dR}(A^\bullet)_{\sigma^i \tau_0}$ is equal to $\Frob$ if $\sigma^{i-1}\tau\in \Phi$, and to $\Ver^{-1}$ otherwise.\footnote{We have that $\Ver:H_1^{\rm dR}(A^\bullet)_{\sigma\tau}\rightarrow H_1^{\rm dR}(A^\bullet)_{\tau}$ is invertible if and only if $\tau\notin \Phi$.}
			Then, there exists a submodule $M\subset H^{\rm dR}_1(A^\bullet)_{\tau_0}\otimes \bar{\FF}_p$ of rank $n/2+1$ satisfying the condition 
			\[
			\widetilde{\Frob}^f((M^{(p^f)}))^\perp\subset M,
			\] 
			where $^\perp$ is taken with respect to the pairing \[\langle-,-\rangle^\bullet: H_1^{\rm dR}(A^\bullet)_{\tau_0}\times H_1^{\rm dR}(A^\bullet)_{\tau_0^*}\rightarrow \bar{\FF}_p.\footnote{For the totally split case, see \cite{XZ2}. To prove that such a submodule $M$ exists, one can argue by induction on $n$ as in \cite{XZ2}.}\]
	The Deligne--Lusztig variety $DL$ is the moduli of all submodules $M$ satisfying the above conditions.
		
		To construct $A$ over $DL^{\perf}$, we use covariant Dieudonn\'e theory and by a theorem of Gabber (see for instance \cite[Theorem D]{Lau}), we only need realize the Dieudonn\'e module of $A$ as a sub-Dieudonn\'e module of $D(A^\bullet)$.

		We define $D(A)_\tau\subset D(A^\bullet)_\tau$ as follows.
		
		\begin{enumerate}
			\item For $\tau\notin \co_{\tau_0}$:
			we define $D(A)_\tau=D(A^\bullet)_\tau$  if $\tau\notin \Phi$, and $D(A)_\tau=pD(A^\bullet)_\tau$  if $\tau\in \Phi$.
			\item For $\tau\in\co_{\tau_0}$:		
			let $\tilde{M}$ be the preimage of $M$ in $D(A^\bullet)_{\tau_0}$ under the $\bmod$ $p$ map, and $\tilde{M}^*$ the preimage of $M^\perp$ in  $D(A^\bullet)_{\tau_0^*}$.
			We define $D(A)_{\tau_0^*}=\Frob(\tilde{M})$ and $D(A)_{\tau_0}=\Ver^{-1}(\tilde{M}^*)$. For $\tau=\sigma^i\tau_0$, where $1\leq i\leq f$ (resp. $f+1\leq i\leq 2f$), we define $D(A)_\tau=\widetilde{\Ver}D(A)_{\sigma\tau}$ inductively from $\tau_0^c=\sigma^f\tau_0$ (resp. $\tau_0=\tau_0^{2f}$), where $\widetilde{\Ver}=\Ver$ if $\tau\notin \Phi$ and $\widetilde{\Ver}=p^{-1}\Ver$ otherwise.
		\end{enumerate}
		
		Note that by definition the submodule $D(A)$ of $D(A^\bullet)$ is invariant under $\Frob$ and $\Ver$. Hence, it is a sub-Dieudonn\'e module.\footnote{One may use the non-emptiness of basic locus due to Wedhorn and Viehmann to prove the existence of such $M$. The argument goes the reverse way: let $A$ be the abelian variety with $\nu(A)=\nu_b$. As in \cite{vollaard}, one can construct an $\Frob$-invariant lattice from $D(A)$ and this lattice recovers $\Lambda_p^\bullet$.}
		Furthemore, the abelian variety $A$ inherits a polarization and additional PEL-structures from those of $A^\bullet$, and satisfies Kottwitz's determinant condition.

		\item Show that step (2) constructs an irreducible component of $\cS_\p(\nu_b)$, or similarly of $RZ$. Indeed, this can be proven by showing that the image of $DL$ has the correct dimension. On one hand, by the same argument as in the proof of \cite[Proposition 2.13]{vollaard} (replacing the $p$-Frobenius by the $q$-Frobenius), the dimension of the Deligne--Lusztig variety $DL$ in step (2) is $n/2-1$. On the other hand, by Theorem \ref{thm_VW}, the $\mu$-ordinary Newton polygon $\nu_o(\co_{\tau_0})$ 
		has break points 
		\[(2f,f), \ (2f(n-1), (n-1)f-1), \ {\rm and} \ (2fn, fn).\] 
		Hence all possible non-supersingular Newton polygons are in one-to-one correspondence with integer points with abscissa $2ft$, for some $t\in \ZZ\cap [1,n/2]$. In particular,   by \cite[Theorem 1.1]{hamacher} (combined with Theorem \ref{thm_VW}) the basic locus (i.e., the supersingular locus) has codim $n/2$, and thus dimension $n/2-1$. 
	\end{enumerate}
	
	To conclude, we observe that $C_p$ is the stabilizer of this irreducible component under  the action of  $I(\QQ_p)$ on $RZ$ arising from its natural action on the associated Dieudonn\'e modules.
\end{proof}

\begin{proof}[Proof of Theorem \ref{prop_irredcomp}]
	By Proposition \ref{prop_double_coset}, the theorem follows once we provide an asymptotic lower bound for $m(C_pK^p)$ which grows to infinity with $p$.
	By \cite[Proposition 2.13]{GHY},\footnote{A mass formula for the stabilizer of a maximal lattice in a quadratic or Hermitian space over a totally real field was first proved by Shimura in \cite{shimura, sh3, sh2}.} 
	$$m(C_pK^p)= c\cdot \lambda_S$$
	where  $\lambda_S=\prod_{p\in S} \lambda_p$,
	for $\lambda_p$ an explicit local factor at $p$ and $S$ the set of finite places $v$ of $\QQ$ where $I_v$ is not isomorphic to $G_v$, and 
	$c={2^{-(n[F:\QQ]+1)}}c'\cdot L(M_I)\cdot \tau(I),$ where 
	\begin{itemize}
		\item $M_I$ is a motive of Artin--Tate type attached to $I$ by Gross, 
		\item $L(M_I)$ is the value of the $L$-function of $M_I$ at $0$, 
		\item $\tau(I)$ is the Tamagawa number of $I$, and 
		\item $c'$ depends only on the non-hyperspecial piece of the level $K$.
	\end{itemize}
	
	Note that in our case $S=\{p\}$ (see Section \ref{ISO}).
	We claim that $c$ is independent of $p$. Indeed, the constant $L(M_I)$ only depends on the quasi-split inner form of $I$ over $\QQ$, which is independent of $p$; by \cite{Kottwitz88}, $\tau(I)$ is independent of $p$ because the center of the neutral connected component of the Langlands dual group of $I$ is independent of $p$; 
	since $K$ is hyperspecial at $p$, then $c'$ is independent of $p$.

	Hence, to conclude, it suffices to prove that the local factor $\lambda_p$ is unbounded as $p$ grows to infinity.
	In  \cite[Formula (2.12)]{GHY}, the local factor $\lambda_p$ is explicitely computed as
	$$\lambda_p=\frac{p^{-N(\overline{G}_p)}\cdot \# \overline{G}_p({\mathbb F}_p)}{{p^{-N(\overline{I}_p)}\cdot \# \overline{I}_p({\mathbb F}_p)}},$$
	where, for $H=G,I$, the group $\overline{H}_p$ denotes the maximal reductive quotient of
	the special fiber of $H_p$, and $N(\overline{H}_p)$ denotes the number of positive roots of $\overline{H}_p$ over $\overline{\mathbb F}_p$. Note that the integral structure of $I_p$ is given by $C_p$.
	
	Assume $p$ is split in $E/F$.  Then, $G$ and $I$ only differ at $\p$. More precisely, $G_{\ZZ_p}=\GG_m\prod_{v|p} G_v$ and $G_v\cong I_v$ for $v\neq \p$. At $\p$, the group $C_\p=\cO^\times_{D_{F_\p,n}}\subset I_\p$ is Iwahori and 
	$\overline{I}_\p$ modulo $\GG_m$ is a totally non-split torus of rank $n-1$. More precisely, $\overline{I}_\p$ is the multiplicative group of the degree $n$ extension of $\cO_F/\p$. 
	Hence, for  $q=\# \cO_F/\p$, we have
	\[
	\lambda_p=\frac{q^{-N(\GL_n)}\# \GL_n(\FF_q)}{\# \overline{I}_\p(\FF_q)}=\frac{q^{(1-n)n/2}\prod_{i=1}^n(q^n-q^{i-1})}{q^n-1}=\prod_{i=2}^n(q^{n-i+1}-1). 
	\]

	Assume $\p$ is inert in $E/F$.  In  Lemma \ref{lem_inert_L} below, we verify that the $\cO_{E}\otimes \ZZ_p$-lattice $\Lambda_p^\bullet$ is maximal.
	Hence, the computation in \cite[\S 3, Table 2]{GHY} applies for
	$n$ even.  For $q=\# \cO_F/\p$, we have \[\lambda_p=(q^{n}-1)/(q+1).\] 
\end{proof}

\begin{lemma}\label{lem_inert_L}
	Maintaining the same assumptions as in Theorem \ref{prop_irredcomp}, suppose  $\p$ is inert in $E/F$ and $n$ is even. 
	Then the $\cO_{E}\otimes \ZZ_p$-lattice $\Lambda_p^\bullet$ is maximal. More precisely, let $\HH$ be the split rank $2$ Hermitian space over $E_v$ with a standard basis $\{e,f\}$ and let $\Delta$ be the maximal lattice $\cO_{E,v}e\oplus \cO_{E,v}f$.
	
	\begin{enumerate}
		\item if $v\neq \p$, then $\Lambda_v^\bullet\cong \Delta^n$ as Hermitian lattices.
		\item if $v=\p$, then, as Hermitian lattices, $\Lambda_v^\bullet\cong \Delta^{n-1}\oplus \cO_D$, where $\cO_D$ is the maximal order of the unique quaternion algebra over $F_v$.\footnote{More precisely, we equip $D$ with the following Hermitian form. As an $E_v$ vector space, write $D=E_v+zE_v$ such that $z^2=\alpha:=\disc(V^\bullet\otimes_E E_v)$ and $xz=zx^c$ for $x\in E_v$. Then the Hermitian form is given by $\langle x_1+zx_2, y_1+zy_2 \rangle=x_1y_1^c-\alpha x_2y_2^c$}
	\end{enumerate}
	
\end{lemma}
\begin{proof}
	Let $\Lambda$ be a maximal lattice containing   $\Lambda^\bullet_p$. We use integral Witt decomposition for the maximal lattice $\Lambda$ to check that $\Lambda$ itself satisfies the duality condition $\Lambda^\vee/\Lambda\simeq \cO_E/\p$. Hence, $\Lambda=\Lambda^\bullet_p$ and it is maximal. 
	(For definition and results on the integral Witt decomposition, we refer to \cite[Chapter 1]{shimura}. Also,  note that since $p$ is unramified in $E/\QQ$, the inverse of the different ideal $\fd^{-1}$ is relatively prime to $p$.)  
\end{proof}

\bibliographystyle{amsplain}
\bibliography{npsecondbib}

\end{document}